\documentclass[11pt]{article}
\usepackage{graphicx}
\usepackage{amsmath}
\usepackage{amssymb}
\usepackage{amsthm}
\usepackage{esint}
\usepackage{mathtools}
\usepackage{mathrsfs}
\usepackage[shortlabels]{enumitem}
\usepackage{xparse}
\usepackage{authblk}
\usepackage{hyperref}

\numberwithin{equation}{section}
\newtheorem{definition}{Definition}[section]
\newtheorem{lemma}[definition]{Lemma}
\newtheorem{proposition}[definition]{Proposition}
\newtheorem{proposition*}[definition]{*Proposition}
\newtheorem{definition*}[definition]{*Definition}
\newtheorem{theorem}[definition]{Theorem}
\newtheorem{example}[definition]{Example}
\newtheorem{corollary}[definition]{Corollary}
\newtheorem{remark}[definition]{Remark}
\newcommand{\sub}{\subseteq}

\newcommand{\R}{\mathbb{R}}

\newcommand{\E}{\mathbb{E}}
\newcommand{\M}{\mathbb{M}}
\newcommand{\pr}{\mathbb{P}}

\newcommand*{\n}[1]{\left\|#1\right\|}
\newcommand*{\inn}[2]{\left \langle #1, #2 \right \rangle}

\newcommand*{\capa}[2]{\mathrm{Cap}\left( #1, #2 \right)}
\newcommand*{\wst}{\xrightharpoonup{*}}

\title{Homogenization for the Poisson Equation in Domains Perforated by Random Closed Sets}
\author{Naoto Sato}
\affil{Graduate School of Science, Tohoku University\\
Sendai 980-8578, Japan\\
email: \href{mailto:sato.naoto.t4@dc.tohoku.ac.jp}{sato.naoto.t4@dc.tohoku.ac.jp}}

\begin{document}

\maketitle
\begin{abstract}
    We study the homogenization of the Poisson equation in randomly perforated domains and obtain the strange term effect in the homogenized equation. The perforations are modeled by rescaled germ-grain processes, and the main assumption is stationarity of the capacities of the holes. We emphasize that the potential in the homogenized equation is constant, despite the possibly nonstationary spatial distribution of the holes. We also establish corrector results.
\end{abstract}
\section{Introduction}
\label{sect: intro}
In this article, we consider the asymptotic behavior as $\varepsilon \searrow 0$ of solutions $u^{\varepsilon}$ to the Poisson equations
\begin{equation}
    \label{eq: poisson intro perforated}
    \begin{cases}
        -\Delta u^{\varepsilon} = f & \text{ in $D^{\varepsilon}$}\\
        u^{\varepsilon} = 0 & \text{ on $\partial D^{\varepsilon}$}
    \end{cases}
\end{equation}
posed on a family of perforated domains $D^{\varepsilon}$. The domain $D^{\varepsilon}$ is obtained by perforating a fixed bounded domain $D \sub \R^{d}, d \ge 3$. One consequence of our main result (Theorem \ref{thm: main thm}) is, roughly speaking, 
\begin{theorem}[a special case of Example \ref{thm: c alpha}]
    \label{thm: intro}
    Let $\{(z_{i}, K_{i})\}_{i = 1}^{\infty}$ be a countable collection of random pairs $(z_{i}, K_{i})$, where $z_{i} \in \R^{d}$ and $K_{i} \sub \R^{d}$ is an arbitrary compact set centered at the origin. Suppose that the number of points $z_{i}$ in any cube has a finite first moment, and that the process $\nu := \{(z_{i}, \mathrm{Cap}K_{i})\}_{ i= 1}^{\infty}$ is stationary, that is, the law of $(z_{i} + \tau, \mathrm{Cap}K_{i})_{i = 1}^{\infty}$ is the same as that of $(z_{i}, \mathrm{Cap}K_{i})_{i = 1}^{\infty}$ for every $\tau \in \R^{d}$. Here, $\mathrm{Cap}K_{i}$ denotes the capacity of $K_{i}$ (see Section \ref{sect: capacity}). Let $W \supseteq D$ be a bounded domain in $\R^{d}$ with $\partial W$ of Lebesgue measure zero and define 
    \[H^{\varepsilon} := \bigcup_{i: \varepsilon z_{i} \in W}(\varepsilon z_{i} +\varepsilon^{\frac{d}{d - 2}}K_{i}),\, D^{\varepsilon} := D \setminus H^{\varepsilon}\]
    (the domain $W$ serves as an observation window of the processes $\{\varepsilon z_{i}\}_{i}$). Suppose further that there exists a nonrandom constant $\alpha > 0$ such that every hole $K_{i}$ satisfies
    \[\frac{\mathrm{Cap}K_{i}}{(\mathrm{diam}K_{i})^{d - 2}} \ge \alpha.\]
    Then we may define a nonnegative random variable $C_{0}$ such that if $C_{0} < \infty$ a.s., then with probability one, the solutions $u^{\varepsilon} \in H^{1}_{0}(D^{\varepsilon})$ to the problems \eqref{eq: poisson intro perforated}, when extended by $0$, converge weakly in $H^{1}_{0}(D)$ to the solution $u \in H^{1}_{0}(D)$ to the problem
    \begin{equation}
        \label{eq: homogenized eq intro}
        \begin{cases}
            -\Delta u + C_{0}u = f & \text{ in $D$}\\
            u = 0 \text{ on $\partial D$}.
        \end{cases}
    \end{equation}
    Further, $C_{0}$ admits an explicit representation \eqref{eq: def of C0} as a conditional expectation of a spatial average of $\mathrm{Cap}K_{i}$. Moreover, if $\nu$ is ergodic, then $C_{0}$ is nonrandom.
\end{theorem}

As can be easily seen from the proof of Theorem \ref{thm: main thm}, the solutions $u^{\varepsilon}$ cannot converge strongly in $H^{1}(D)$ unless $C_{0} = 0$. However, as observed by Cioranescu and Murat \cite{cioranescu1997strange} in the deterministic case, we can use the functions $\hat{e}^{\varepsilon}$ that are constructed in Section \ref{sect: proof of the main thm} and are independent of the source term $f$ to make ``oscillate'' the function $u$ and to establish the strong convergence $\n{u^{\varepsilon} - (1 - \hat{e}^{\varepsilon})u}_{H^{1}_{0}(D)} \to 0$ a.s. and in $L^{2}(\pr)$ (Proposition \ref{thm: corrector}). In other words, we also prove corrector results.

Although homogenization problems related to \eqref{eq: poisson intro perforated} are classical (see, e.g., \cite{cioranescu1997strange,dalmaso1987wiener, dalmaso1994new, kac1974probabilistic, marchenko1974boundary, marchenko2006homogenization, papanicolaou1980diffusion, rauch1975potential}), these problems are still actively investigated when the domains $D^{\varepsilon}$ are randomly perforated (\cite{caffarelli2009random, calvo2015homogenization,calvojurado2016homogenization,giunti2018homogenization,giunti2021convergence,hoefer2024convergence, scardia2025homogenisation}). In \cite{caffarelli2009random}, Caffarelli and Mellet considered obstacle problems on $D^{\varepsilon}$. While the holes are located on the lattice $\varepsilon\mathbb{Z}^{d}$, the restrictions on the shape of holes are imposed only on their diameters and capacities. On the other hand, Giunti, H{\"o}fer, and Vel{\'a}zquez \cite{giunti2018homogenization} showed the same homogenization result as Theorem \ref{thm: intro} holds when the centers of holes are distributed according to a stationary point process on $\R^{d}$, if holes $K_{i}$ are all balls.

The restriction on the shape of holes in \cite{giunti2018homogenization} is due to the use of explicit formulae for the capacitary potentials of balls. We closely follow the discussion in \cite{giunti2018homogenization} while replacing parts of hard arguments involving the explicit formulae with softer ones that make use of the maximum principle and Lemma \ref{thm: Maz'ya lemma} by Maz'ja \cite{MR817985}, and thereby succeed in allowing the homogenization to occur under  fairly mild conditions on both the distribution and the geometry of holes $K_{i}$. We also mention that the set of assumptions on point processes that we use to generate the set of holes is slightly different from that in \cite{giunti2018homogenization}. Instead, we follow the framework presented in \cite{Bas25}. This modification allows us to use a strong law of large numbers for stationary marked point processes with arbitrary mark space (Lemma \ref{thm: averaging h over A}).

This article is organized as follows. In Section \ref{sect: Preliminaries}, we collect some preliminaries needed to state the main theorem (Theorem \ref{thm: main thm}). In Section \ref{sect: statement of the main result}, we present the precise statement of Theorem \ref{thm: main thm} and provide an example. Since Theorem \ref{thm: main thm} can be regarded as a norm resolvent convergence of Dirichlet Laplacians, by the standard argument we also obtain a convergence result for the solutions to the heat equations in perforated domains (Corollary \ref{thm: convergence of heat eq}). Section \ref{sect: proof of the main thm} is devoted to the proof of the main theorem. In Section \ref{sect: corrector}, we present corrector results. Finally, in Section \ref{sect: appendix}, we gather auxiliary results to make the presentation self-contained.
\section{Preliminaries}
\label{sect: Preliminaries}
\subsection{Notations}
\label{sect: notations}
\begin{itemize}
    \item We denote by $\mathcal{B}(E)$ the Borel $\sigma$-algebra of a topological space $E$.
    \item We denote by $\#A$ the number of points in a set $A$.
    \item We denote by $\mathcal{L}^{d}$ the $d$-dimensional Lebesgue measure. For $\mathcal{L}^{d}$-measurable $A \sub \R^{d}$, $|A| := \mathcal{L}^{d}(A)$.
    \item We denote by $\mathcal{C}$ the space of compact subsets in $\R^{d}$ equipped with the Hausdorff metric.
    \item For $x \in \R^{d}$ and $r \ge 0$, $U_{r}(x)$ and $B_{r}(x)$ denote the open and closed balls of radius $r$ centered at $x$, respectively.
    \item For an open $U \sub \R^{d}$, $C^{\infty}_{c}(U)$ denotes the class of $\R$-valued infinitely differentiable functions on $U$ with compact support.
    \item For $u \in H^{1}_{0}(U)$ with $U \sub \R^{d}$ open, the zero extension $\tilde{u} \in H^{1}(\R^{d})$ of $u$ is also denoted by $u$.
    \item For $a, b \in [-\infty, \infty]$, we use the notation $a \lesssim b$ if $a \le C b$ for a constant $C$ that depends only on the dimension $d$. Moreover, we use the notation $a \asymp b$ if both $a \lesssim b$ and $b \lesssim a$ hold.
    \item Throughout this paper we fix a probability space $(\Omega, \mathcal{F}, \pr)$. By random elements in a measurable space $E$, we mean $\mathcal{F}$-measurable maps from $\Omega$ to $E$.
    \item $\psi[X]$ denotes the random measure determined by the relation \eqref{eq: def of psi X}. For details, see the next section.
\end{itemize}
\subsection{Capacity}
\label{sect: capacity}
As we noted in Introduction, the potential $C_{0}$ in the homogenized equation \eqref{eq: homogenized eq intro} is described by a sort of average of the capacity of holes.
\begin{definition}
    \label{def: capacity}
    For a compact $F \sub \R^{d}$ and nonempty open $U \sub \R^{d}$ with $F \sub U$, define
    \begin{equation}
        \label{def: mathcal K_ F, U}
        \mathcal{K}_{F, U} := \{\varphi \in C_{c}^{\infty}(U): 0 \le \varphi \le 1, \varphi = 1 \text{ on $F$}\}
    \end{equation}
    and
    \begin{equation}
        \label{def: K_ F, U}
        K_{F, U} := \text{ the completion of $\mathcal{K}_{F, U}$ with respect to the norm $\n{\nabla (\cdot)}_{L^{2}(U)}$}.
    \end{equation}

    If $U$ is bounded, then $K_{F, U}$ is nonempty convex closed in the Hilbert space $H^{1}_{0}(U)$, and hence the problem
    \begin{equation}
        \label{def: cap pot in terms of minimization of energy}
        \text{minimize} \int_{U}|\nabla u|^{2}\,dx \text{ subject to } u \in K_{F, U}
    \end{equation}
    admits the unique solution $e \in K_{F, U}$. We call this $e$ the capacitary potential of $F$ in $U$. The minimum $\int_{U}|\nabla e|^{2}$ is called the capacity of $F$ in $U$ and is denoted by $\capa{F}{U}$.

    If $U = \R^{d}$, then $K_{F, \R^{d}}$ is nonempty convex closed in the Hilbert space $\{u \in L^{2^{*}}: \nabla u \in L^{2}\}$ equipped with the inner product $(u, v) \mapsto \inn{\nabla u}{\nabla v}_{L^{2}}$. The unique minimizer $e$ of the problem \eqref{def: cap pot in terms of minimization of energy} is called the capacitary potential of $F$. The minimum is called the capacity of $F$ and is denoted by $\mathrm{Cap}F$.
\end{definition}

By definition, capacity is invariant under rigid motion, and has the following scaling properties:
\begin{equation}
    \label{eq: scaling property of capacity}
    \mathrm{Cap}(rF, rU) = r^{d - 2}\mathrm{Cap}(F, U)
\end{equation}
and
\begin{equation}
    \mathrm{Cap}(rF) = r^{d - 2}\mathrm{Cap}F
\end{equation}
for every $r > 0$, compact $F \sub \R^{d}$, and bounded open $U \sub \R^{d}$ with $U \supseteq F$.

The next lemma is the key to identifying $C_{0}$:
\begin{lemma}[{\cite[Lemma 2.5.1/1]{MR817985}}]
    \label{thm: Maz'ya lemma}
    For $\lambda > 1$, define 
    \begin{equation}
        f_d(\lambda) := \begin{cases}
            \frac{1}{\lambda - 1}\left( 1 + 2\log \lambda \right) & \text{ if $d = 3$}\\
            \frac{2}{d - 3}\frac{1}{\lambda - 1} & \text{ if $d \ge 4$}.
        \end{cases}
    \end{equation}
    Then for every compact $F \sub \R^{d}$ and $R > r > 0$ such that $F \sub B_{r}(0)$, 
    \begin{equation}
        \capa{F}{U_{R}(0)} \leq \left(1 + f_{d}\left(\frac{R}{r}\right)\right)\mathrm{Cap}(F).
    \end{equation}
\end{lemma}

\begin{remark}
    There is an obvious typo in the statement of \cite[Lemma 2.5.1/1]{Mazya2011sobolevspaces}. The correct statement is written in the older edition, \cite[Lemma 2.5.1/1]{MR817985}.
\end{remark}
\subsection{Marked point processes}
\label{sect: mpp}
Following \cite[Section 2.1]{Bas25}, we define marked point processes.
\begin{definition}
    \label{def: mpp}
    Let $\mathcal{K}$ be a complete separable metric space.
    \begin{enumerate}[(1)]
        \item A subset $Y \sub \R^{d} \times \mathcal{K}$ is called admissible if it is a graph of a $\mathcal{K}$-valued function whose domain $\mathrm{dom}Y \sub \R^{d}$ is locally finite.
        \item The class of all admissible sets in $\R^{d} \times \mathcal{K}$ is denoted by $\mathbb{M}^{d}_{\mathcal{K}}$.
        \item The class $\mathbb{M}^{d}_{\mathcal{K}}$ is endowed with the $\sigma$-algebra
        \[\mathcal{M}^{d}_{\mathcal{K}} := \sigma \left( Y \mapsto \# (Y \cap A): A \in \mathcal{B}(\R^{d} \times \mathcal{K}) \right).\]
        \item A random element in $(\mathbb{M}^{d}_{\mathcal{K}}, \mathcal{M}^{d}_{\mathcal{K}})$ is called a marked point process (in short MPP) on $\R^{d}$ with marks in $\mathcal{K}$.
        \item Given an MPP $X$ on $\R^{d}$ with marks in $\mathcal{K}$, $\mathrm{dom}X$ is identified with the MPP with marks in a one-point set and is called the ground process of $X$.
    \end{enumerate}
\end{definition}

\begin{remark}
    The measurable space $(\mathbb{M}^{d}_{\mathcal{K}}, \mathcal{M}^{d}_{\mathcal{K}})$ that we have just defined may be identified with a class $\mathcal{N}^{\#g}_{\R^{d} \times \mathcal{K}}$ of Borel measures on $\R^{d} \times \mathcal{K}$, which is defined in \cite[Definition 9.1.II]{daley2008introduction} via the map $Y \mapsto \#(Y \cap \cdot)$. This correspondence is a known fact (see, e.g., \cite[Lemma 3.1.4]{schneider2008stochastic}). For the reader's convenience we state and prove the precise statement in Proposition \ref{thm: correspondence between N hash g and admissible sets}. Thus, we can use the theory of marked point processes developed in \cite{daley2008introduction}.
\end{remark}
For each $\tau \in \R^{d}$, we define the shift $S_{\tau}: \mathbb{M}^{d}_{\mathcal{K}} \to \mathbb{M}^{d}_{\mathcal{K}}$ by 
\[S_{\tau}Y := \{(y - \tau, k): (y, k) \in Y\}.\]
The map $\R^{d} \times \mathbb{M}^{d}_{\mathcal{K}} \ni (\tau, Y) \mapsto S_{\tau}Y \in \mathbb{M}^{d}_{\mathcal{K}}$ is known to be jointly measurable (see, for example, \cite[Lemma 12.1.I]{daley2008introduction}). Hence, we may introduce the notions of stationarity and ergodicity of marked point processes:
\begin{definition}
    Let $X$ be a marked point process on $\R^{d}$ with marks in $\mathcal{K}$.
    \begin{enumerate}[(1)]
        \item The MPP $X$ is called stationary if every shift $\mathbb{M}^{d}_{\mathcal{K}} \to \mathbb{M}^{d}_{\mathcal{K}}$ preserves the law of $X$.
        \item We define the invariant sub-$\sigma$-algebgra $\mathcal{I}_{X}$ of $\mathcal{F}$ by
        \[\mathcal{I}_{X} := \{\{X \in \mathcal{A}\}: \mathcal{A} \in \mathcal{M}^{d}_{\mathcal{K}}, \pr\left( X \in \mathcal{A} \triangle S_{\tau}\mathcal{A} \right) = 0 \text{ for every $\tau \in \R^{d}$}\}.\]
        \item $X$ is said to be ergodic if $\mathcal{I}_{X}$ is trivial, i.e., $\pr(A) = 0\text{ or }1$ for every $A \in \mathcal{I}_{X}$.
    \end{enumerate}
\end{definition}

Now suppose that $X$ is stationary and that $\E[\#(X \cap A)] < \infty$ for every bounded Borel $A \sub \R^{d} \times \mathcal{K}$. Then by \cite[Lemma 12.2.III]{daley2008introduction} there is an $\mathcal{I}_{X}$-measurable random measure $\psi$ on $\mathcal{K}$ such that for every nonnegative Borel measurable function $f$ on $\mathcal{K}$ and every $A \in \mathcal{B}(\R^{d})$ with $0 < |A| < \infty$, we have
\begin{equation}
    \label{eq: def of psi X}
    \E\left[\frac{1}{|A|}\sum_{(x, k) \in X, x \in A}f(k)\Big| \mathcal{I}_{X}\right] = \int_{\mathcal{K}} f \,d\psi \text{ a.s.}
\end{equation}
We denote the random measure $\psi$ as above by $\psi[X]$.

\section{Statement of the main result}
\label{sect: statement of the main result}
Let $W \sub \R^{d}, d \ge 3$ be bounded open and $|\partial W| = 0$. For $\varepsilon > 0$, let $X^{\varepsilon}$ be an MPP on $\R^{d}$ with marks in $\mathcal{C}$ and define the set of ``holes'' by 
\[H^{\varepsilon} := \bigcup_{(z, K_{z}) \in X^{\varepsilon}, \varepsilon z \in W}(\varepsilon z + \varepsilon^{\frac{d}{d - 2}}K_{z}).\]
Intuitively, $H^{\varepsilon}$ is a rescaled germ-grain model observed in a window $W$. 
Throughout this paper we assume that $X^{\varepsilon}$ satisfies the following:
\begin{itemize}
    \item There exist MPPs $Z^{\varepsilon}$ on $\R^{d}$ with marks in
    \begin{equation}
        \label{def: mark space Q}
        \mathcal{Q} := \{(K, \rho) \in \mathcal{C} \times \R_{\ge 0}: K \sub B_{\rho}(0)\}
    \end{equation}
    such that a.s., for every $\varepsilon \in ]0, 1]$, 
    \begin{equation}
        X^{\varepsilon} = \{(z, K): (z, K, \rho) \in Z^{\varepsilon}\}.
    \end{equation}
    This ensures that each hole is contained in a ball and enables us to utilize the results in \cite{giunti2018homogenization}.
    \item The MPP
    \begin{equation}
        N := \{(z, \mathrm{Cap}K, \rho): (z, K, \rho) \in Z^{\varepsilon}\}
    \end{equation}
    is independent of $\varepsilon$.
    \item We also assume that $N$ is stationary, $\E[\#N(A)] < \infty$ whenever $A \sub \R^{d} \times \R_{\ge 0}^{2}$ is bounded Borel, and 
    \begin{equation}
        \label{eq: rho has finite d -2 moment}
        \int_{\R_{\ge 0}^{2}}\rho^{d - 2}\,\psi[N](dk, d\rho) < \infty \text{ a.s.}
    \end{equation}
\end{itemize}
Under these assumptions, we have
\begin{theorem}
    \label{thm: main thm}
    With probability one, the following hold: let $D \sub W$ be nonempty open and a sequence $f^{\varepsilon} \in H^{-1}(D)$ converge to $f \in H^{-1}(D)$ strongly in $H^{-1}(D)$. Put $D^{\varepsilon}  := D \setminus H^{\varepsilon}$. Then the solutions $u^{\varepsilon} \in H^{1}_{0}(D^{\varepsilon})$ to 
    \begin{equation}
        \label{eq: Poisson in perforated domains}
        \begin{cases}
            -\Delta u^{\varepsilon} = f^{\varepsilon} & \text{ in $H^{-1}(D^{\varepsilon})$}\\
            u^{\varepsilon} \in H^{1}_{0}(D^{\varepsilon})
        \end{cases}
    \end{equation}
    weakly converge in $H^{1}_{0}(D)$ to the solution $u \in H^{1}_{0}(D)$ to 
    \begin{equation}
        \label{eq: homogenized equation}
        \begin{cases}
        -\Delta u + C_{0}u = f & \text{in $H^{-1}(D)$}\\
        u \in H^{1}_{0}(D),
    \end{cases}
    \end{equation}
    where $C_{0}$ is a random variable given by 
    \begin{equation}
        \label{eq: def of C0}
        C_{0} := \int_{\R_{\ge 0}}k\,\psi[N](dk, d\rho) = \E\left[ \frac{1}{|A|}\sum_{(z, K) \in X^{1}, z \in A}\mathrm{Cap}K\Big|\mathcal{I}_{N} \right]
    \end{equation}
    for any $A \in \mathcal{B}(\R^{d})$ with $0 < |A| < \infty$. In particular, $C_{0}$ is nonrandom if $N$ is ergodic.
\end{theorem}

\begin{corollary}
    \label{thm: convergence of heat eq}
    Let $r^{\varepsilon}: L^{2}(D) \to L^{2}(D^{\varepsilon})$ denote the restriction operator and $i^{\varepsilon}: L^{2}(D^{\varepsilon}) \to L^{2}(D)$ the extension by zero. Further, let $\Delta_{D^{\varepsilon}}$ and $\Delta_{D}$ denote the Dirichlet Laplacians on $D^{\varepsilon}$ and $D$, respectively. Then for every $T_{0} > 0$, we have 
    \begin{equation}
        \sup_{T_{0} \le t < \infty}\n{i^{\varepsilon}e^{t\Delta_{D^{\varepsilon}}}r^{\varepsilon} - e^{-C_{0}t}e^{t \Delta_{D}}}_{L^{2}(D) \to L^{2}(D)} \to 0
    \end{equation}
    a.s. and in $L^{p}(\pr)$ for every $1 \le p < \infty$.
\end{corollary}

\begin{proof}
    By Theorem \ref{thm: main thm} and \cite[Proposition 4.1.1]{daners2008domain}, $i^{\varepsilon}(- \Delta_{D^{\varepsilon}})^{-1}r^{\varepsilon} \xrightarrow{B(L^{2}(D), L^{2}(D))} (-\Delta_{D} + C_{0})^{-1}$ a.s. Hence, by a standard approximation argument (see, e.g., \cite[Theorem 4.2.9]{post2012spectral}), we have that
    \[\n{e^{t \Delta_{D^{\varepsilon}}}r^{\varepsilon} - r^{\varepsilon}e^{-C_{0}t}e^{t \Delta_{D}}}_{L^{2}(D) \to L^{2}(D^{\varepsilon})} \to 0 \text{ a.s.}\]
    for each $t \in ]0, \infty[$. Further, 
    \begin{equation*}
        \begin{split}
            \n{(1 - i^{\varepsilon}r^{\varepsilon})e^{-C_{0}t}e^{t\Delta_{D}}}_{L^{2}(D) \to L^{2}(D)} &\le \n{1 - i^{\varepsilon}r^{\varepsilon}}_{L^{\infty}(D) \to L^{2}(D)}\n{e^{-C_{0}t}e^{t\Delta_{D}}}_{L^{2}(D) \to L^{\infty}(D)}\\
            &\lesssim |D \setminus D^{\varepsilon}|^{\frac{1}{2}}e^{-C_{0}t}t^{-\frac{d}{4}}
        \end{split}
    \end{equation*}
    and
    \[|D \setminus D^{\varepsilon}| \lesssim \sum_{(z, k, \rho) \in N, \varepsilon z \in W}(\varepsilon^{\frac{d}{d - 2}}\rho)^{d} \le \left( \max_{(z, k, \rho) \in N, \varepsilon z \in W} \varepsilon^{\frac{d}{d - 2}}\rho\right)^{2}\varepsilon^{d}\sum_{(z, k, \rho) \in N, \varepsilon z \in W}\rho^{d - 2}.\]
    By Lemma \ref{thm: averaging h over A} and Lemma \ref{thm: dividing balls into good and bad parts} below, we see that the right-hand side of the last display converges to $0$ a.s. Therefore, 
    \[\n{i^{\varepsilon}e^{t\Delta_{D^{\varepsilon}}}r^{\varepsilon} - e^{-C_{0}t}e^{t\Delta_{D}}}_{L^{2}(D) \to L^{2}(D)} \to 0 \text{ a.s.}\]
    for each $t \in ]0, \infty[$. Note also that 
    \begin{equation*}
        \begin{split}
            &\left| \n{i^{\varepsilon}e^{t\Delta_{D^{\varepsilon}}}r^{\varepsilon} - e^{-C_{0}t}e^{t \Delta_{D}}}_{L^{2}(D) \to L^{2}(D)} - \n{i^{\varepsilon}e^{s\Delta_{D^{\varepsilon}}}r^{\varepsilon} - e^{-C_{0}s}e^{s \Delta_{D}}}_{L^{2}(D) \to L^{2}(D)} \right|\\
            & \le \n{i^{\varepsilon}(e^{t \Delta_{D^{\varepsilon}}} - e^{s\Delta_{D^{\varepsilon}}})r^{\varepsilon}}_{L^{2}(D) \to L^{2}(D)} + \n{e^{-t(\Delta_{D} - C_{0})} - e^{-s(\Delta_{D} - C_{0})}}_{L^{2}(D) \to L^{2}(D)}\\
            &\le 2\sup_{\lambda_{1}(D) \le x < \infty}|e^{-tx} - e^{-sx}|
        \end{split}
    \end{equation*}
    for every $s, t > 0$, where $\lambda_{1}(D)$ denotes the first eigenvalue of $-\Delta_{D}$. Since the family $(e^{-tx})_{T_{0} \le t < \infty} \sub C_{0}([\lambda_{1}(D), \infty[)$ is totally bounded for every $T_{0} > 0$, the last two displays show the desired a.s. convergence. The convergence in $L^{p}(\pr)$ follows from the a.s. convergence by dominated convergence.
\end{proof}

\begin{example}
    \label{thm: c alpha}
    For a nonempty $K \in \mathcal{C}$, let $c(K)$ denote the circumcenter of $K$, that is, the center of the smallest ball containing $K$. For $\alpha \ge 0$, Put
    \[\mathcal{C}_{\alpha} := \left\{ K \in \mathcal{C}: K \neq \varnothing, c(K) = 0, \frac{\mathrm{Cap}K}{(\mathrm{diam}K)^{d - 2}} \ge \alpha \right\}.\]
    Let $X^{\varepsilon}$ be MPPs with marks in $\mathcal{C}_{\alpha}, \alpha > 0$ nonrandom such that the MPP
    \[\nu := \{(z, \mathrm{Cap}K): (z, K) \in X^{\varepsilon}\}\]
    is independent of $\varepsilon$. Suppose further that $\nu$ is stationary, $\E[\#(\nu \cap A)] < \infty$ whenever $A \sub \mathcal{B}(\R^{d} \times \R_{\ge 0})$ is bounded, and $\int_{\R_{\ge 0}}k\,\psi[\nu](dk) < \infty$ a.s. Then $X^{\varepsilon}$ satisfies assumptions listed above, and thus a.s., the solution to \eqref{eq: Poisson in perforated domains} converges weakly in $H^{1}_{0}(D)$ to the solution to \eqref{eq: homogenized equation} with
    \[C_{0} = \int_{\R_{\ge 0}}k\,\psi[\nu](dk).\]
    Moreover, if $\nu$ is ergodic, then $C_{0}$ is nonrandom, and
    \[C_{0} = \E\left[ \frac{1}{|A|}\sum_{(z, K) \in X^{1}, z \in A}\mathrm{Cap}K \right]\]
    for any $A \in \mathcal{B}(\R^{d})$ with $0 < |A| < \infty$.

    Indeed, put $Z^{\varepsilon} := \{(z, K, \left( \frac{\mathrm{Cap}K}{\alpha} \right)^{\frac{1}{d - 2}}): (z, K) \in X^{\varepsilon}\}$. Then 
    \[N = \left\{\left(z, k, \left( \frac{k}{\alpha} \right)^{\frac{1}{d - 2}}\right): (z, k) \in \nu\right\}\]
    and it is stationary. Also, it is easily seen that $\mathcal{I}_{N} = \mathcal{I}_{\nu}$.
\end{example}

\begin{remark}
    Let $V: W \to \R_{\ge 0}$ be bounded continuous and put 
    \[H^{\varepsilon}_{V} := \bigcup_{(z, K_{z}) \in X^{\varepsilon}, \varepsilon z \in W}(\varepsilon z + \varepsilon^{\frac{d}{d - 2}}V(\varepsilon z)K_{z}), \, D^{\varepsilon}_{V} := D \setminus H^{\varepsilon}_{V}.\]
    Then with probability one, the solutions $u^{\varepsilon}_{V} \in H^{1}_{0}(D^{\varepsilon}_{V})$ to 
    \begin{equation*}
        \begin{cases}
            -\Delta u^{\varepsilon}_{V} = f^{\varepsilon} & \text{ in $D^{\varepsilon}_{V}$}\\
            u^{\varepsilon}_{V} = 0 & \text{ on $\partial D^{\varepsilon}_{V}$}
        \end{cases}
    \end{equation*}
    converge weakly in $H^{1}_{0}(D)$ to the solution to
    \begin{equation*}
        \begin{cases}
            -\Delta u_{V} + C_{0}V^{d - 2}u = f & \text{ in $D$}\\
            u_{V} = 0 & \text{ on $\partial D$}.
        \end{cases}
    \end{equation*}
    To see this, observe first that we may assume that $V \le 1$ on $W$. Indeed, replace $X^{\varepsilon}$ with $\{(z, \n{V}_{\infty}K): (z, K) \in X^{\varepsilon}\}$. Then for $z \in I^{\varepsilon}_{g}$ and $M >1$, let $\hat{e}^{\varepsilon}_{z, V}$ be the capacitary potential of $\varepsilon z + V(\varepsilon z)\varepsilon^{\frac{d}{d - 2}}K^{\varepsilon}_{z}$ in $U^{\varepsilon}_{z}$ and $\hat{e}^{\varepsilon}_{z, M, V}$ the capacitary potential of $\varepsilon z + \frac{M}{\rho_{z} \vee M}\varepsilon^{\frac{d}{d - 2}}V(\varepsilon z)K^{\varepsilon}_{z}$ in $U^{\varepsilon}_{z}$, where $I^{\varepsilon}_{g}, U^{\varepsilon}_{z}$ are as in Section \ref{sect: lemma from giunti}. Also, define $\hat{e}^{\varepsilon}_{g, V} := \sum_{z \in I^{\varepsilon}_{g}}\hat{e}^{\varepsilon}_{z, V}$, $\hat{e}^{\varepsilon}_{V} := \hat{e}^{\varepsilon}_{g, V} + \hat{e}^{\varepsilon}_{b}$, and $\hat{e}^{\varepsilon}_{g, M, V} := \sum_{z \in I^{\varepsilon}_{g, M}}\hat{e}^{\varepsilon}_{z, M, V}$. Then Lemma \ref{thm: weak convergence of the modulus of grad egM} with $\hat{e}^{\varepsilon}_{g, M}$ and $\chi_{W}$ replaced with $\hat{e}^{\varepsilon}_{g, M, V}$ and $V^{d - 2}$, respectively holds true. Consequently, by the same arguments as in Section \ref{sect: proof of construction of osc fcts}, we may show that $\hat{e}^{\varepsilon}$ satisfy \eqref{eq: ee1}, \eqref{eq: H3 for e}, and \eqref{eq: H5 for e} with $\int_{W}v \, dx$ replaced with $\int_{W}v V^{d - 2} \, dx$.
\end{remark}

\section{Proof of Theorem \ref{thm: main thm}}
\label{sect: proof of the main thm}
To prove Theorem \ref{thm: main thm}, we construct suitable oscillating test functions, just as \cite{cioranescu1997strange,giunti2018homogenization} do.

\begin{lemma}
    \label{thm: construction of osc fct}
    There is a family $(\hat{e}^{\varepsilon})_{\varepsilon \in ]0, 1]}$ of random functions such that
    \begin{equation}
        \label{eq: ee1}
        \hat{e}^{\varepsilon} \in K_{H^{\varepsilon}, \{d(\cdot, W) < \varepsilon\}},
    \end{equation}
    \begin{equation}
        \label{eq: H3 for e}
        \hat{e}^{\varepsilon} \xrightharpoonup{H^{1}(\R^{d})} 0 \text{ weakly a.s. as $\varepsilon \searrow 0$},
    \end{equation}
    and a.s., for every family $(v^{\varepsilon})_{\varepsilon \in ]0, 1]}$ such that $v^{\varepsilon} \rightharpoonup v$ weakly in $H^{1}_{0}(W)$ and $v^{\varepsilon} \in H^{1}_{0}(W \setminus H^{\varepsilon})$, it holds that 
    \begin{equation}
        \label{eq: H5 for e}
        \inn{\Delta \hat{e}^{\varepsilon}}{v^{\varepsilon}}_{H^{-1}(W), H^{1}_{0}(W)} \to \int_{W}v \,dx \int_{\R_{\ge 0}^{2}}k \,\psi[N](dk, d\rho)
    \end{equation}
    as $\varepsilon \searrow 0$. 
\end{lemma}

\begin{proof}[Proof of Theorem \ref{thm: main thm}]
    Fix $\omega \in \Omega$ such that convergences \eqref{eq: H3 for e} and \eqref{eq: H5 for e} hold. Since $\n{\nabla u^{\varepsilon}}_{L^{2}(D)} \le C(d, D)\n{f^{\varepsilon}}_{H^{-1}(D)}$ and $f^{\varepsilon} \to f$ in $H^{-1}(D)$, there is a subsequence, still denoted by $\varepsilon$, and $u \in H^{1}_{0}(D)$ such that $u^{\varepsilon} \rightharpoonup u$ weakly in $H^{1}_{0}(D)$. Thus the only task left is to identify $u$.

    Let $\varphi \in C^{\infty}_{c}(D)$ and put $w^{\varepsilon} := 1 - \hat{e}^{\varepsilon}$. Then $\varphi w^{\varepsilon} \in H^{1}_{0}(W \setminus H^{\varepsilon})$. Hence, 
    \[\int_{D^{\varepsilon}}\nabla u^{\varepsilon} \cdot \nabla (\varphi w^{\varepsilon}) = \inn{f^{\varepsilon}}{\varphi w^{\varepsilon}}_{H^{-1}(D), H^{1}_{0}(D)},\]
    or equivalently, 
    \begin{equation}
        \label{eq: tested by vf we}
        \int_{D}w^{\varepsilon}\nabla u^{\varepsilon} \cdot \nabla \varphi + \int_{D} \nabla (\varphi u^{\varepsilon}) \cdot \nabla w^{\varepsilon} - \int_{D}u^{\varepsilon}\nabla \varphi \cdot \nabla w^{\varepsilon} = \inn{f^{\varepsilon}}{\varphi w^{\varepsilon}}_{H^{-1}(D), H^{1}_{0}(D)}.
    \end{equation}
    Thanks to \eqref{eq: H3 for e}, \eqref{eq: H5 for e}, the weak convergence $u^{\varepsilon} \rightharpoonup u$ in $H^{1}_{0}(D)$, and the strong convergence $f^{\varepsilon} \to f$ in $H^{-1}(D)$, we may calculate the limit of every term in \eqref{eq: tested by vf we} and conclude that 
    \[\int_{D}\nabla u \cdot \nabla \varphi + \left( \int_{\R_{\ge 0}^{2}}k \,\psi[N](dk, d\rho) \right)\int_{D}u\varphi = \inn{f}{\varphi}_{H^{-1}(D), H^{1}_{0}(D)}.\]
    Since $\varphi \in C^{\infty}_{c}(D)$ is arbitrary, this shows that $u$ solves \eqref{eq: homogenized equation}.
\end{proof}

The rest of this section is devoted to the proof of Lemma \ref{thm: construction of osc fct}.

Let $\mathcal{K}$ be a complete separable metric space. For $\delta > 0$, define the thinning maps $\mathcal{T}_{\delta}, \mathcal{T}^{\delta}: \mathbb{M}^{d}_{\mathcal{K}} \to \mathbb{M}^{d}_{\mathcal{K}}$ by
\begin{equation}
    \label{def: thinning map}
    \mathcal{T}_{\delta}(Y) := \{(z, \rho) \in Y: \min_{(z', \rho') \in Y, z \neq z'}|z - z'| < \delta\}
\end{equation}
and $\mathcal{T}^{\delta}(Y) := Y \setminus \mathcal{T}_{\delta}(Y)$. In \cite[Proposition 7.6]{Bas25}, it is shown that each $\mathcal{T}_{\delta}$ is measurable (strictly speaking, the author of \cite{Bas25} deals with the case $\mathcal{K} = \R_{\ge 0}$, but the same proof applies to arbitrary $\mathcal{K}$). Further, $\mathcal{T}^{\cdot}(\cdot), \mathcal{T}_{\cdot}(\cdot)$ are jointly measurable. Indeed, for every $Y \in \mathcal{M}^{d}_{\mathcal{K}}$ and bounded Borel $A \sub \R^{d} \times \mathcal{K}$, the function $\R_{> 0} \ni \delta \mapsto \#(\mathcal{T}_{\delta}(Y) \cap A) \in \mathbb{Z}_{\ge 0}$ is left continuous since $\mathrm{dom}Y$ is locally finite. Given a stationary MPP $X$, $\mathcal{T}_{\delta}(X), \mathcal{T}^{\delta}(X)$ are stationary \cite[Proposition 7.1]{Bas25} since $\mathcal{T}_{\delta}, \mathcal{T}^{\delta}$ commute with the shifts $S_{\tau}$, $\tau \in \R^{d}$, and we write
\begin{equation}
    \psi_{\delta}[X] := \psi[\mathcal{T}_{\delta}(X)], \psi^{\delta}[X] := \psi[\mathcal{T}^{\delta}(X)]
\end{equation}
whenever $\psi[X]$ is well-defined.
\subsection{Lemmas from \texorpdfstring{\cite{giunti2018homogenization}}{}}
\label{sect: lemma from giunti}
The lemmas contained in this subsection, which we present with proofs for the reader's convenience, are essentially proved in \cite{giunti2018homogenization} from a different set of assumptions from ours, but the proofs below are almost the same as theirs: we use Lemmas \ref{thm: averaging h over A}, \ref{thm: giunti lem5_2}, instead of \cite[Lemma 5.1, 5.2]{giunti2018homogenization}, respectively.

Let $\Xi$ be a stationary MPP on $\R^{d}$ with marks in $\R_{\ge 0}$ such that $\E[\#(\Xi \cap A)] < \infty$ whenever $A \in \mathcal{B}(\R^{d} \times \R_{\ge 0})$ is bounded and
\begin{equation}
        \label{cond: moment condition on radii}
        \int_{\R_{\ge 0}} \rho^{d - 2}\,\psi[\Xi](d\rho) < \infty \text{ a.s.}
    \end{equation}
Note that this ensures that $\psi[\Xi](\R_{\ge 0}) < \infty$ a.s. The ground process of $\Xi$ is denoted by $\Phi$. For $z \in \Phi$, the unique $\rho \in \R_{\ge 0}$ such that $(z, \rho) \in \Xi$ is denoted by $\rho_{z}$.
\begin{lemma}[cf. {\cite[Lemma 4.2]{giunti2018homogenization}}]
    \label{thm: dividing balls into good and bad parts}
    There exist r.v.s $(r_{\varepsilon})_{0 < \varepsilon \le 1} \sub ]0, 1]$ and $I^{\varepsilon}_{b} \sub \Phi \cap \varepsilon^{-1}W$ such that defining 
    \begin{equation}
        \label{def: Heb}
        H^{\varepsilon}_{b} := \bigcup_{z \in I^{\varepsilon}_{b}}B_{\varepsilon^{\frac{d}{d - 2}}\rho_{z}}(\varepsilon z),
    \end{equation}
    \begin{equation}
        \label{def: Deb}
        D^{\varepsilon}_{b} := \bigcup_{z \in I^{\varepsilon}_{b}}U_{2\varepsilon^{\frac{d}{d - 2}}\rho_{z}}(\varepsilon z),    
    \end{equation}

    \begin{equation}
        \label{def: Ieg}
        I^{\varepsilon}_{g} := (\Phi \cap \varepsilon^{-1}W)\setminus I^{\varepsilon}_{b},        
    \end{equation}
    \begin{equation}
        \label{def: Heg}
        H^{\varepsilon}_{g} := \bigcup_{z \in I^{\varepsilon}_{g}}B_{\varepsilon^{\frac{d}{d - 2}}\rho_{z}}(\varepsilon z),
    \end{equation}
    we have the following:
    \begin{enumerate}
        \item the following a.s. convergences take place as $\varepsilon \searrow 0$:
        \begin{equation}
        \label{eq: re to 0}
        r_{\varepsilon} \to 0,
    \end{equation}
    \begin{equation}
        \label{eq: e^d Ieb to 0}
        \varepsilon^{d}\#I^{\varepsilon}_{b} \to 0,
    \end{equation}
    \begin{equation}
        \label{eq: cap Heb Deb to 0}
        \mathrm{Cap}(H^{\varepsilon}_{b}, D^{\varepsilon}_{b}) \to 0,
    \end{equation}
    \begin{equation}
        \label{eq: cap Deb to 0}
        \mathrm{Cap}(D^{\varepsilon}_{b}) \to 0,
    \end{equation}
    and for every $\delta > 0$,
    \begin{equation}
        \label{eq: points in thinned processs near Deb vanish asymptotically}
        \varepsilon^{d}\#\{z \in \mathcal{T}^{2\delta}(\Phi): \varepsilon z \in W, d(\varepsilon z, D^{\varepsilon}_{b}) \le \delta \varepsilon\} \to 0;
    \end{equation}
    \item a.s., for every $\varepsilon \in ]0, 1]$,
    \begin{equation}
        \label{eq: balls on good part are disjoint}
        (B_{\varepsilon^{\frac{d}{d - 2}}\rho_{z}}(\varepsilon z))_{z \in I^{\varepsilon}_{g}} \text{ is disjoint};
    \end{equation}
    \item a.s., for every $\varepsilon \in ]0, 1]$ and $z \in I^{\varepsilon}_{g}$, 
    \begin{equation}
        \label{eq: centers on good part are well separated}
        \min_{\substack{z_{1} \in \Phi,\\ z \neq z_{1}}}|\varepsilon z - \varepsilon z_{1}| \ge 2\varepsilon r_{\varepsilon},
    \end{equation}
    \begin{equation}
        \label{eq: radii on good part are small}
        \varepsilon^{\frac{d}{d - 2}}\rho_{z} < \frac{1}{2}\varepsilon r_{\varepsilon},
    \end{equation}
    and
    \begin{equation}
        \label{eq: good and bad parts are well separated}
        d(H^{\varepsilon}_{g}, D^{\varepsilon}_{b}) \ge \frac{1}{2}\varepsilon r_{\varepsilon}.
    \end{equation}
    \end{enumerate}
\end{lemma}

\begin{proof}
    Fix $\alpha \in ]0, \frac{d}{d - 2}[$. For $\varepsilon \in]0, 1]$, set 
    \begin{equation}
        \label{def: re}
        r_{\varepsilon} := \left( \left( \varepsilon^{\frac{d}{d - 2}}\max_{z \in \Phi \cap \varepsilon^{-1}W}\rho_{z} \right)^{\frac{1}{d}} \wedge 1 \right) \vee \varepsilon^{\alpha}.
    \end{equation}

        Put 
        \begin{equation}
            \label{def: Fe}
            F^{\varepsilon} := \{z \in \Phi \cap \varepsilon^{-1}W: \varepsilon^{\frac{d}{d - 2}}\rho_{z} \ge \varepsilon\}.
        \end{equation}
        If $F^{\varepsilon} = \varnothing$, then $r_{\varepsilon} \le \varepsilon^{\frac{1}{d}} \vee \varepsilon^{\alpha}$. If $F^{\varepsilon} \neq \varnothing$, then 
        \begin{equation}
            \label{eq: evaluating re}
            r_{\varepsilon} \le \left( \varepsilon^{\frac{d}{d - 2}}\max_{z \in F^{\varepsilon}}\rho_{z} \right)^{\frac{1}{d}} \vee \varepsilon^{\alpha} \le \left( \varepsilon^{d}\sum_{z \in F^{\varepsilon}}(\rho_{z})^{d - 2} \right)^{\frac{1}{d(d - 2)}} \vee \varepsilon^{\alpha}.
        \end{equation}
        Further, 
        \[\varepsilon^{d}\#F^{\varepsilon} \le \varepsilon^{d}\sum_{z \in F^{\varepsilon}}(\varepsilon^{\frac{2}{d - 2}}\rho_{z})^{d - 2} \le \varepsilon^{2}\cdot \varepsilon^{d}\sum_{z \in \Phi \cap \varepsilon^{-1}W}\rho_{z}^{d - 2} \to 0 \text{ a.s.}\]
        by Lemma \ref{thm: averaging h over A}. Hence, by Lemma \ref{thm: giunti lem5_2}, the right-hand side of \eqref{eq: evaluating re} vanishes a.s. as $\varepsilon \searrow 0$, and \eqref{eq: re to 0} is established.

    For $\varepsilon \in ]0, 1]$, define
    \begin{equation}
        \label{def: eta e}
        \eta_{\varepsilon} := \varepsilon r_{\varepsilon},
    \end{equation} 
    \begin{equation}
        \label{def: Jeb}
        J^{\varepsilon}_{b} := \{z \in \Phi \cap \varepsilon^{-1}W: 2\varepsilon^{\frac{d}{d - 2}}\rho_{z} \ge \eta_{\varepsilon}\},
    \end{equation}
    \begin{equation}
        \label{def: Keb}
        K^{\varepsilon}_{b} := \mathcal{T}_{2r_{\varepsilon}}(\Phi) \cap \varepsilon^{-1}W \setminus J^{\varepsilon}_{b},
    \end{equation}
    \begin{equation}
        \label{def: tilde Heb}
        \tilde{H}^{\varepsilon}_{b} := \bigcup_{z \in J^{\varepsilon}_{b}}B_{2\varepsilon^{\frac{d}{d - 2}}\rho_{z}}(\varepsilon z),
    \end{equation}
    \begin{equation}
        \label{def: tilde Ieb}
        \tilde{I}^{\varepsilon}_{b} := \{z \in \Phi \setminus (J^{\varepsilon}_{b} \cup K^{\varepsilon}_{b}): \varepsilon z \in W, B_{\eta_{\varepsilon}}(\varepsilon z) \cap \tilde{H}^{\varepsilon}_{b} \neq \varnothing\},
    \end{equation}
    \begin{equation}
        \label{def: Ieb}
        I^{\varepsilon}_{b} := \tilde{I}^{\varepsilon}_{b} \cup J^{\varepsilon}_{b} \cup K^{\varepsilon}_{b}.
    \end{equation}
    Then by subadditivity of capacity, 
    \begin{equation*}
        \label{eq: evaluate cap Heb in Deb}
        \begin{split}
            \capa{H^{\varepsilon}_{b}}{D^{\varepsilon}_{b}} &\le \capa{B_{\varepsilon^{\frac{d}{d - 2}}\rho_{z}}(\varepsilon z)}{U_{2\varepsilon^{\frac{d}{d - 2}}\rho_{z}}(\varepsilon z)}\\
            &\lesssim \varepsilon^{d}\sum_{z \in I^{\varepsilon}_{b}}(\rho_{z})^{d - 2}
        \end{split}
    \end{equation*}
    and 
    \[\mathrm{Cap}(D^{\varepsilon}_{b}) \lesssim \varepsilon^{d}\sum_{z \in I^{\varepsilon}_{b}}(\rho_{z})^{d - 2}.\]
    Thus, by virtue of Lemma \ref{thm: giunti lem5_2}, to show \eqref{eq: cap Heb Deb to 0} and \eqref{eq: cap Deb to 0}, it suffices to show the convergence \eqref{eq: e^d Ieb to 0}.

    By \eqref{eq: re to 0} and Lemma \ref{thm: averaging h over A},
        \begin{equation}
            \label{eq: jeb to 0}
            \begin{split}
                \varepsilon^{d}\#J^{\varepsilon}_{b} &\lesssim \varepsilon^{d}\sum_{z \in J^{\varepsilon}_{b}}(\varepsilon^{\frac{d}{d - 2}}\rho_{z}\eta_{\varepsilon}^{-1})^{d}\\
                &\le r_{\varepsilon}^{-d}\left(\max_{z \in \Phi \cap \varepsilon^{-1}W}\varepsilon^{\frac{d}{d - 2}}\rho_{z}\right)^{2}\sum_{z \in \Phi \cap \varepsilon^{-1}W}\varepsilon^{d}\rho_{z}^{d - 2}\\
                &\le r_{\varepsilon}^{d}\varepsilon^{d}\sum_{z \in \Phi \cap \varepsilon^{-1}W}\rho_{z}^{d - 2}\\
                &\to 0 \text{ a.s.}
            \end{split}
        \end{equation}

        Also, by Lemma \ref{thm: continuity of thinning}, 
        \begin{equation}
            \label{eq: Keb to 0}
           \begin{split}
            \limsup_{\varepsilon}\varepsilon^{d}\#K^{\varepsilon}_{b} &\le \limsup_{\varepsilon}\varepsilon^{d}\#\left( \mathcal{T}_{m^{-1}}(\Phi) \cap (\varepsilon^{-1}W) \right)\\
            &\le |W|\psi_{m^{-1}}[\Xi](\R_{\ge 0})\\
            &\to 0 \text{ a.s. as $m \to \infty$.}
           \end{split}
        \end{equation}

        Next, we show that $\varepsilon^{d}\#\tilde{I}^{\varepsilon}_{b} \to 0$ a.s.
        If $z \in \tilde{I}^{\varepsilon}_{b}$, then
        \[\min_{\substack{z_{1} \in \Phi,\\ z \neq z_{1}}}|\varepsilon z - \varepsilon z_{1}| \ge 2\eta_{\varepsilon},\]
        since $z \notin K^{\varepsilon}_{b} \cup J^{\varepsilon}_{b}$.
        Hence, $(B_{\eta_{\varepsilon}}(\varepsilon z))_{z \in \tilde{I}^{\varepsilon}_{b}}$ is essentially disjoint, in the sense that the intersection of any two of the members of that family is of $\mathcal{L}^{d}$-null. Thus,
        \begin{equation}
            \label{eq: tildeIeb to 0}
            \begin{split}
                \varepsilon^{d}\#\tilde{I}^{\varepsilon}_{b} &\lesssim \varepsilon^{d}\eta_{\varepsilon}^{-d}\sum_{z \in \tilde{I}^{\varepsilon}_{b}}|B_{\eta_{\varepsilon}(\varepsilon z)}|\\
                &= r_{\varepsilon}^{-d}\left| \bigcup_{z \in \tilde{I}^{\varepsilon}_{b}}B_{\eta_{\varepsilon}}(\varepsilon z)\right|\\
                &\le r_{\varepsilon}^{-d}\left|\{d(\cdot, \tilde{H}^{\varepsilon}_{b}) \le 2\eta_{\varepsilon}\}\right|\\
                &= r_{\varepsilon}^{-d}\left| \bigcup_{z \in J^{\varepsilon}_{b}}B_{2\varepsilon^{\frac{d}{d - 2}}\rho_{z} + 2\eta_{\varepsilon}}(\varepsilon z)\right|\\
                &\lesssim r_{\varepsilon}^{-d}\sum_{z \in J^{\varepsilon}_{b}}\left( \varepsilon^{\frac{d}{d - 2}}\rho_{z} \right)^{d}\\
                &\lesssim r_{\varepsilon}^{-d}\left( \max_{z \in \Phi \cap \varepsilon^{-1}W}\varepsilon^{\frac{d}{d - 2}}\rho_{z} \right)^{2}\varepsilon^{d}\sum_{z \in J^{\varepsilon}_{b}}\rho^{d - 2}\\
                &\le r_{\varepsilon}^{d}\varepsilon^{d}\sum_{z \in J^{\varepsilon}_{b}}(\rho_{z})^{d - 2}\\
                &\to 0 \text{ a.s.}
            \end{split}
        \end{equation}
        by Lemma \ref{thm: giunti lem5_2}.

    By \eqref{eq: jeb to 0}, \eqref{eq: Keb to 0}, and \eqref{eq: tildeIeb to 0}, $\varepsilon^{d}\#I^{\varepsilon}_{b} \to 0$ a.s.

        We now turn to the proof of \eqref{eq: points in thinned processs near Deb vanish asymptotically}. For $\varepsilon \in ]0, 1]$ and $\delta > 0$, define 
        \begin{equation}
            \label{def: Eed}
            E^{\varepsilon}_{\delta} := \{z \in I^{\varepsilon}_{g}: d(\varepsilon z, \tilde{H}^{\varepsilon}_{b}) \le \delta \varepsilon\}
        \end{equation}
        and 
        \begin{equation}
            \label{def: Ced}
            C^{\varepsilon}_{\delta} := \left\{z \in I^{\varepsilon}_{g} \setminus \mathcal{T}_{2\delta}(\Phi): d\left( \varepsilon z, \bigcup_{z_{1} \in K^{\varepsilon}_{b} \cup \tilde{I}^{\varepsilon}_{b}}B_{2\varepsilon^{\frac{d}{d - 2}}\rho_{z_{1}}}(\varepsilon z_{1}) \right) \le \delta \varepsilon\right\}.
        \end{equation}
        Then 
        \[\{z \in \mathcal{T}^{2\delta}(\Phi): \varepsilon z \in W, d(\varepsilon z, D^{\varepsilon}_{b}) \le \delta \varepsilon\} \sub I^{\varepsilon}_{b} \cup E^{\varepsilon}_{\delta} \cup C^{\varepsilon}_{\delta}.\]

        If $z \in C^{\varepsilon}_{\delta}$, then there is $z_{1} \in K^{\varepsilon}_{b} \cup \tilde{I}^{\varepsilon}_{b}$ such that
        \[d\left( \varepsilon z, B_{2\varepsilon^{\frac{d}{d - 2}}\rho_{z_{1}}}(\varepsilon z_{1}) \right) \le \delta \varepsilon,\]
        and thus 
        \[|\varepsilon z - \varepsilon z_{1}| \le \delta \varepsilon + 2\varepsilon^{\frac{d}{d - 2}}\rho_{z_{1}} \le \delta \varepsilon + \eta_{\varepsilon}.\]
        Hence, $C^{\varepsilon}_{\delta} \sub \mathcal{T}_{\delta + r_{\varepsilon}}(\Phi) \setminus \mathcal{T}_{2 \delta}(\Phi)$. By \eqref{eq: re to 0}, this implies that a.s., 
        \begin{equation}
            \label{eq: Ced is empty for small e}
            C^{\varepsilon}_{\delta} = \varnothing
        \end{equation}
        for sufficiently small $\varepsilon$.

        Next, we show that
        \begin{equation}
            \label{eq: Eed to 0}
            \varepsilon^{d}\#E^{\varepsilon}_{\delta} \to 0 \text{ a.s.}
        \end{equation}
        as $\varepsilon \searrow 0$. If $z \in E^{\varepsilon}_{\delta}$, then $d(\varepsilon z, B_{2\varepsilon^{\frac{d}{d - 2}}\rho_{z_{1}}}(\varepsilon z_{1})) \le \delta \varepsilon$ for some $z_{1} \in J^{\varepsilon}_{b}$, and thus
        \[B_{\eta_{\varepsilon}}(\varepsilon z) \sub B_{2\varepsilon^{\frac{d}{d - 2}}\rho_{z_{1}} + \delta \varepsilon + \eta_{\varepsilon}}(\varepsilon z_{1}).\]
        Hence, a.s., for sufficiently small $\varepsilon$ that $(\varepsilon^{\frac{d}{d - 2}}\max_{z \in \Phi \cap \varepsilon^{-1}W}\rho_{z})^{\frac{1}{d}} \le r_{\varepsilon} \le \delta$,
        \begin{equation*}
            \label{eq: evaluate Eed}
            \begin{split}
                \varepsilon^{d}\#E^{\varepsilon}_{\delta} &\lesssim \varepsilon^{d}\eta_{\varepsilon}^{-d}\sum_{z \in E^{\varepsilon}_{\delta}}|B_{\eta_{\varepsilon}}(\varepsilon z)|\\
                &= r_{\varepsilon}^{-d}\left| \bigcup_{z \in E^{\varepsilon}_{\delta}}B_{\eta_{\varepsilon}}(\varepsilon z) \right|\\
                &\le r_{\varepsilon}^{-d}\left| \bigcup_{z_{1} \in J^{\varepsilon}_{b}}B_{2\varepsilon^{\frac{d}{d - 2}}\rho_{z} + \delta \varepsilon + \eta_{\varepsilon}}(\varepsilon z_{1}) \right|\\
                &\lesssim r_{\varepsilon}^{-d}\sum_{z_{1} \in J^{\varepsilon}_{b}}\left(\left( 2 + \frac{2\delta}{r_{\varepsilon}} + 2 \right) \varepsilon^{\frac{d}{d - 2}}\rho_{z_{1}} \right)^{d}\\
                &\lesssim \delta^{d}r_{\varepsilon}^{-2d}\sum_{z_{1} \in J^{\varepsilon}_{b}}\left( \varepsilon^{\frac{d}{d - 2}}\rho_{z_{1}} \right)^{d}\\
                &\le \delta^{d}\varepsilon^{d}\sum_{z_{1} \in J^{\varepsilon}_{b}}(\rho_{z_{1}})^{d - 2}.
            \end{split}
        \end{equation*}
        Thus, by \eqref{eq: jeb to 0} and Lemma \ref{thm: giunti lem5_2}, $\varepsilon^{d}\#E^{\varepsilon}_{\delta} \to 0$ a.s.

        Finally, we show \eqref{eq: good and bad parts are well separated}. Notice that for every $z \in I^{\varepsilon}_{g}$ and $z_{1} \in I^{\varepsilon}_{b}$,  
        \begin{equation}
            \label{eq: balls in good and bad parts are disjoint}
            B_{\eta_{\varepsilon}}(\varepsilon z) \cap B_{2\varepsilon^{\frac{d}{d - 2}}\rho_{z_{1}}}(\varepsilon z_{1}) = \varnothing.
        \end{equation}
        Indeed, since $z \notin \tilde{I}^{\varepsilon}_{b}$, \eqref{eq: balls in good and bad parts are disjoint} holds if $z_{1} \in J^{\varepsilon}_{b}$. Also, since $z \notin K^{\varepsilon}_{b} \cup J^{\varepsilon}_{b}$, $|\varepsilon z - \varepsilon z_{1}| \ge 2\eta_{\varepsilon}$, and thus \eqref{eq: balls in good and bad parts are disjoint} holds if $z_{1} \in I^{\varepsilon}_{b} \setminus J^{\varepsilon}_{b}$. By \eqref{eq: balls in good and bad parts are disjoint}, 
    \begin{equation*}
        \label{evaluate dist betweed Heg and Deb}
        \begin{split}
            d\left( H^{\varepsilon}_{g}, D^{\varepsilon}_{b} \right) &= \min_{\substack{z \in I^{\varepsilon}_{g},\\ z_{1} \in I^{\varepsilon}_{b}}}d\left( B_{\varepsilon^{\frac{d}{d - 2}}\rho_{z}}(\varepsilon z), B_{2\varepsilon^{\frac{d}{d - 2}}\rho_{z_{1}}}(\varepsilon z_{1}) \right)\\
            &\ge \min_{z \in I^{\varepsilon}_{g}}(\eta_{\varepsilon} - \varepsilon^{\frac{d}{d - 2}}\rho_{z})\\
            &\ge \frac{1}{2}\eta_{\varepsilon}.
        \end{split}
    \end{equation*}
\end{proof}

\begin{lemma}[cf.{\cite[Lemma 3.1]{giunti2018homogenization}}]
    \label{thm: giunti 3_1}
    Let $M > 1$. For $\varepsilon \in ]0, 1]$ and $z \in I^{\varepsilon}_{g}$, define 
    \begin{equation}
        \label{def: dezr}
        d^{\varepsilon}_{z} := \min\left\{ \varepsilon, \frac{1}{2}\min_{\substack{z_{1} \in \Phi,\\ z \neq z_{1}}}|\varepsilon z - \varepsilon z_{1}|, d(\varepsilon z, D^{\varepsilon}_{b}) \right\}
    \end{equation}
    and
    \begin{equation}
        \label{def: Uezr}
        U^{\varepsilon}_{z} := U_{d^{\varepsilon}_{z}}(\varepsilon z).
    \end{equation}
    Then for a.e. $\omega \in \Omega$ and every $\varepsilon \in ]0, 1]$ and $z \in I^{\varepsilon}_{g}$, 
    \begin{equation}
        \label{eq: evaluate dezr from below}
        2\varepsilon^{\frac{d}{d - 2}}\rho_{z} \le d^{\varepsilon}_{z}.
    \end{equation}
    Then define 
    \begin{equation}
        \label{def: cap pot of ball in Uezr}
        e^{\varepsilon}_{z} := \text{the capacitary potential of $B_{\varepsilon^{\frac{d}{d - 2}}\rho_{z}}(\varepsilon z)$ in $U^{\varepsilon}_{z}$},
    \end{equation}
    \begin{equation}
        \label{def: cap pot of truncated ball in Uezr}
        e^{\varepsilon}_{z, M} := \text{the capacitary potential of $B_{\varepsilon^{\frac{d}{d - 2}}(\rho_{z} \wedge M)}(\varepsilon z)$ in $U^{\varepsilon}_{z}$},
    \end{equation}
    \begin{equation}
        \label{def: osc fct for balls on good part}
        e^{\varepsilon}_{g} := \sum_{z \in I^{\varepsilon}_{g}}e^{\varepsilon}_{z},
    \end{equation}
    \begin{equation}
        \label{def: IegM}
        I^{\varepsilon}_{g, M} := \left\{ z \in I^{\varepsilon}_{g}: d^{\varepsilon}_{z} \ge \frac{\varepsilon}{M} \right\},
    \end{equation}
    \begin{equation}
        \label{def: truncated osc fct for balls on good part}
        e^{\varepsilon}_{g, M} := \sum_{z \in I^{\varepsilon}_{g, M}}e^{\varepsilon}_{z, M},
    \end{equation}
    and 
    \begin{equation}
        \label{def: mueM}
        \mu^{\varepsilon}_{M} := -\sum_{z \in I^{\varepsilon}_{g, M}}\frac{\partial e^{\varepsilon}_{z, M}}{\partial n}\Big|_{\partial U^{\varepsilon}_{z}}\delta_{\partial U^{\varepsilon}_{z}}.
    \end{equation}
    Then the following a.s. convergences take place as $\varepsilon \searrow 0$:
    \begin{equation}
        \label{eq: IegM to N2/M asymptotically}
        \varepsilon^{d}\#\left( \mathcal{T}^{\frac{2}{M}}(\Phi) \cap \varepsilon^{-1}W \setminus I^{\varepsilon}_{g, M} \right) \to 0
    \end{equation}
    and
    \begin{equation}
        \label{eq: osc fct for balls on good part to 0 in weak H1}
        e^{\varepsilon}_{g} \rightharpoonup 0 \text{ weakly in $H^{1}(\R^{d})$}.
    \end{equation}
\end{lemma}

\begin{proof}
    For every $z \in I^{\varepsilon}_{g}$, we have, by Lemma \ref{thm: dividing balls into good and bad parts},
    \[2\varepsilon^{\frac{d}{d - 2}}\rho_{z} \le \eta_{\varepsilon} \le \varepsilon,\]
    \[d\left( \varepsilon z, D^{\varepsilon}_{b} \right) \ge \varepsilon^{\frac{d}{d - 2}}\rho_{z} + d\left( H^{\varepsilon}_{g}, D^{\varepsilon}_{b} \right) \ge 2\varepsilon^{\frac{d}{d - 2}}\rho_{z},\]
    and 
    \[\frac{1}{2}\min_{\substack{z_{1} \in \Phi,\\ z \neq z_{1}}}|\varepsilon z - \varepsilon z_{1}| \ge \eta_{\varepsilon} \ge 2\varepsilon^{\frac{d}{d - 2}}\rho_{z}.\]
    Hence, \eqref{eq: evaluate dezr from below} holds.

    By the definition \eqref{def: dezr} of $d^{\varepsilon}_{z}$,
    \begin{equation}
        \label{eq: evaluate the difference between Ieg and N2/M}
        \begin{split}
            \mathcal{T}^{\frac{2}{M}}(\Phi) \cap \varepsilon^{-1}W \setminus I^{\varepsilon}_{g, M} &\sub \left\{ z \in \mathcal{T}^{\frac{2}{M}}(\Phi) \cap I^{\varepsilon}_{g}: \varepsilon z \in W, d^{\varepsilon}_{z} < \frac{\varepsilon}{M} \right\} \cup I^{\varepsilon}_{b}\\
            &\sub \left\{ z \in \mathcal{T}^{\frac{2}{M}}(\Phi): \varepsilon z \in W, d\left( \varepsilon z, D^{\varepsilon}_{b} \right) < \frac{\varepsilon}{M} \right\} \cup I^{\varepsilon}_{b}.
        \end{split}
    \end{equation}
    Hence, by \eqref{eq: points in thinned processs near Deb vanish asymptotically} and \eqref{eq: e^d Ieb to 0}, the convergence \eqref{eq: IegM to N2/M asymptotically} takes place a.s.

    For $\varepsilon \in ]0, 1]$, by the Poincar{\'e} inequality, 
    \begin{equation}
        \label{eq: evaluate L2 norm of osc fct on good part}
        \int_{\R^{d}}|e^{\varepsilon}_{g}|^{2} = \sum_{z \in I^{\varepsilon}_{g}}\int_{U^{\varepsilon}_{z}}|e^{\varepsilon}_{z}|^{2} \lesssim \sum_{z \in I^{\varepsilon}_{g}}\left( d^{\varepsilon}_{z} \right)^{2}\int_{U^{\varepsilon}_{z}}|\nabla e^{\varepsilon}_{z}|^{2} \le \varepsilon^{2}\int_{\R^{d}}|\nabla e^{\varepsilon}_{g}|^{2}.
    \end{equation}
    Also, by \eqref{eq: evaluate dezr from below},
    \begin{equation}
        \label{eq: evaluate the energy of osc fct on good part}
        \begin{split}
            \int_{\R^{d}}|\nabla e^{\varepsilon}_{g}|^{2} &= \sum_{z \in I^{\varepsilon}_{g}}\capa{B_{\varepsilon^{\frac{d}{d - 2}}\rho_{z}}(\varepsilon z)}{U^{\varepsilon}_{z}}\\
            &\le \sum_{z \in \Phi \cap \varepsilon^{-1}W}\varepsilon^{d}\rho_{z}^{d - 2}\mathrm{Cap}\left(B_{\frac{1}{2}}(0), U_{1}(0)\right).
        \end{split}
    \end{equation}
    From \eqref{eq: evaluate L2 norm of osc fct on good part} and \eqref{eq: evaluate the energy of osc fct on good part}, we have \eqref{eq: osc fct for balls on good part to 0 in weak H1}.
\end{proof}

\subsection{Proof of Lemma \ref{thm: construction of osc fct}}
\label{sect: proof of construction of osc fcts}
First, let us make some definitions.
\begin{itemize}
    \item Let $\Phi$ denote the ground process of $N$, and write 
\begin{equation*}
    Z^{\varepsilon} = \{(z, K^{\varepsilon}_{z}, \rho_{z}): z \in \Phi\},\, N = \{(z, k_{z}, \rho_{z}): z \in \Phi\}.
\end{equation*}

\item Put $\Xi := \{(z, \rho_{z}): z \in \Phi\}$. Since marks of $Z^{\varepsilon}$ belong to $\mathcal{Q}$,
\[\E[\#(\Xi \cap ([0, 1]^{d} \times [0, M]))] = \E[\#(N \cap ([0, 1]^{d} \times [0, \mathrm{Cap}B_{M}(0)] \times [0, M]))]\]
for every $M \in \R_{\ge 0}$. Thus, $\E[\#(\Xi \cap A)] < \infty$ for every bounded $A \in \mathcal{B}(\R^{d} \times \R_{\ge 0})$. Let us then use the notations in Lemma \ref{thm: dividing balls into good and bad parts} and \ref{thm: giunti 3_1} for the MPP $\Xi$.

\item For $\varepsilon \in ]0, 1]$, let $\hat{e}^{\varepsilon}_{b}$ be the capacitary potential of $\hat{H}^{\varepsilon}_{b} := \bigcup_{z \in I^{\varepsilon}_{b}}(\varepsilon z + \varepsilon^{\frac{d}{d - 2}}K^{\varepsilon}_{z})$ in $D^{\varepsilon}_{b}$. Then by \eqref{eq: cap Heb Deb to 0}, $\hat{e}^{\varepsilon}_{b} \to 0$ strongly in $H^{1}(\R^{d})$ a.s.

\item For $z \in I^{\varepsilon}_{g}$, let $\hat{e}^{\varepsilon}_{z}$ be the capacitary potential of $\varepsilon z + \varepsilon^{\frac{d}{d - 2}}K^{\varepsilon}_{z}$ in $U^{\varepsilon}_{z}$, and put 
\begin{equation}
    \hat{e}^{\varepsilon}_{g} := \sum_{z \in I^{\varepsilon}_{g}}\hat{e}^{\varepsilon}_{z}
\end{equation}
and
\begin{equation}
    \label{def: osc fct e}
    \hat{e}^{\varepsilon} := \hat{e}^{\varepsilon}_{g} + \hat{e}^{\varepsilon}_{b}.
\end{equation}
\end{itemize}
    
\begin{proof}[Proof of Lemma \ref{thm: construction of osc fct}]
    Since $\hat{e}^{\varepsilon}_{b} \in K_{\hat{H}^{\varepsilon}_{b}, D^{\varepsilon}_{b}}$ and $\hat{e}^{\varepsilon}_{b} \to 0$ strongly in $H^{1}(\R^{d})$, it suffices to show the following:
    \begin{equation}
        \label{eq: eg is 1 on holes and 0 outside D < e - clDeb}
        \hat{e}^{\varepsilon}_{g} \in K_{\bigcup_{z \in I^{\varepsilon}_{g}}(\varepsilon z + \varepsilon^{\frac{d}{d - 2}}K^{\varepsilon}_{z}), \{d(\cdot, W) < \varepsilon\} \setminus \mathrm{cl}(D^{\varepsilon}_{b})},
    \end{equation}
    \begin{equation}
        \label{eq: eeg to 0 in H1 weak}
        \hat{e}^{\varepsilon}_{g} \xrightharpoonup{H^{1}(\R^{d})} 0 \text{ weakly a.s.},
    \end{equation}
    and \eqref{eq: H5 for e} holds with $\hat{e}^{\varepsilon}$ replaced by $\hat{e}^{\varepsilon}_{g}$.

    By the definition \eqref{def: dezr} of $d^{\varepsilon}_{z}$, $U^{\varepsilon}_{z} \cap \mathrm{cl}D^{\varepsilon}_{b} = \varnothing$ and $U^{\varepsilon}_{z} \sub U_{\varepsilon}(\varepsilon z)$ for $z \in I^{\varepsilon}_{g}$. Hence, \eqref{eq: eg is 1 on holes and 0 outside D < e - clDeb} holds true.

    For $\varepsilon \in ]0, 1]$, $M > 1$, and $z \in I^{\varepsilon}_{g, M}$, define
    \begin{equation}
        \hat{e}^{\varepsilon}_{z, M} := \text{ the capacitary potential of $\varepsilon z + \frac{M}{\rho_{z} \vee M}\varepsilon^{\frac{d}{d - 2}}K^{\varepsilon}_{z}$ in $U^{\varepsilon}_{z}$}.
    \end{equation}
    For $M > 1$, put 
    \begin{equation}
        \hat{e}^{\varepsilon}_{g, M} := \sum_{z \in I^{\varepsilon}_{g, M}}\hat{e}^{\varepsilon}_{z, M}
    \end{equation}
    and
    \begin{equation}
        \hat{\mu}^{\varepsilon}_{M} := -\sum_{z \in I^{\varepsilon}_{g, M}}\frac{\partial \hat{e}^{\varepsilon}_{z, M}}{\partial n}\Big|_{\partial U^{\varepsilon}_{z}}\delta_{\partial U^{\varepsilon}_{z}}.
    \end{equation}
    By the maximum principle, 
    \begin{equation}
        \label{eq: maximum principle cap pot}
        0 \le \hat{e}^{\varepsilon}_{z, M} \le \hat{e}^{\varepsilon}_{z} \le e^{\varepsilon}_{z}.
    \end{equation}
    Hence, 
    \begin{equation*}
        0 \le \hat{e}^{\varepsilon}_{g, M} \le \hat{e}^{\varepsilon}_{g} \le e^{\varepsilon}_{g}.
    \end{equation*}
    Also, 
    \begin{equation*}
        \begin{split}
            \int_{\R^{d}}|\nabla \hat{e}^{\varepsilon}_{g}|^{2} &= \sum_{z \in I^{\varepsilon}_{g}}\capa{\varepsilon z + \frac{M}{\rho_{z} \vee M}K^{\varepsilon}_{z}}{U^{\varepsilon}_{z}}\\
            &\le \sum_{z \in I^{\varepsilon}_{g}}\capa{B_{\varepsilon^{\frac{d}{d - 2}}\rho_{z}}(\varepsilon z)}{U^{\varepsilon}_{z}}\\
            &= \int_{\R^{d}}|\nabla e^{\varepsilon}_{g}|^{2}.
        \end{split}
    \end{equation*}
    Therefore, by \eqref{eq: osc fct for balls on good part to 0 in weak H1}, 
    \begin{equation}
        \label{eq: egM, eg to 0 in weak H1}
        \hat{e}^{\varepsilon}_{g, M}, \hat{e}^{\varepsilon}_{g} \xrightharpoonup{H^{1}(\R^{d})} 0 \text{ weakly a.s.},
    \end{equation}
    and \eqref{eq: eeg to 0 in H1 weak} is established.

    Finally, we prove \eqref{eq: H5 for e}. The strategy is along the same line of that of the proof of \cite[Lemma 3.1]{giunti2018homogenization}: if $v^{\varepsilon} \in H^{1}_{0}(W \setminus H^{\varepsilon})$ and $v^{\varepsilon} \rightharpoonup v \in H^{1}_{0}(W)$ weakly, then we may write 
    \[\inn{\Delta \hat{e}^{\varepsilon}_{g}}{v^{\varepsilon}}_{H^{-1}, H^{1}_{0}} = \inn{\hat{\mu}^{\varepsilon}_{M}}{v^{\varepsilon}}_{H^{-1}, H^{1}_{0}} - \inn{\nabla \hat{e}^{\varepsilon}_{g} - \nabla \hat{e}^{\varepsilon}_{g, M}}{\nabla v^{\varepsilon}}_{L^{2}(W)}.\]
    So, it suffices to show that a.s., 
    \begin{equation}
        \label{eq: mueM converges}
        \hat{\mu}^{\varepsilon}_{M} \xrightarrow[\varepsilon \searrow 0]{H^{-1}(\R^{d})} \chi_{W}\int_{\R_{\ge 0}^{2}}\left( \frac{M}{\rho \vee M} \right)^{d - 2}k \,\psi^{\frac{2}{M}}[N](dk, d\rho),
    \end{equation}
    \begin{equation}
        \label{eq: lim mueM to mu}
        \int_{\R_{\ge 0}^{2}}\left( \frac{M}{\rho \vee M} \right)^{d - 2}k \,\psi^{\frac{2}{M}}[N](dk, d\rho) \xrightarrow[M \to \infty]{} \int_{\R_{\ge 0}^{2}}k \,\psi[N](dk, d\rho),
    \end{equation}
    and 
    \begin{equation}
        \label{eq: nabla eg - egM}
        \limsup_{\varepsilon} \n{\nabla \hat{e}^{\varepsilon}_{g} - \nabla \hat{e}^{\varepsilon}_{g, M}}_{L^{2}(W)} \xrightarrow[M \to \infty]{} 0.
    \end{equation}

    Since 
    \[\varepsilon z + \frac{M}{\rho_{z} \vee M}\varepsilon^{\frac{d}{d - 2}}K^{\varepsilon}_{z} \sub B_{\varepsilon^{\frac{d}{d - 2}}(\rho_{z} \wedge M)}(\varepsilon z),\]
    $0 \le \hat{\mu}^{\varepsilon}_{M} \le \mu^{\varepsilon}_{M}$ by the maximum principle 
    (see \cite[Figure 4]{cioranescu1997strange}). Recall also that, given $0 < r < R < \infty$, the capacitary potential $e$ of $B_{r}(0)$ in $U_{R}(0)$ admits the explicit formula
    \begin{equation*}
        e(x) = 
        \begin{cases*}
            1 & \text{ if $|x| \le r$}\\
            \frac{\left( \frac{1}{|x|} \right)^{d - 2} - \left( \frac{1}{R} \right)^{d - 2}}{\left( \frac{1}{r} \right)^{d - 2} - \left( \frac{1}{R} \right)^{d - 2}} & \text{ if $r \le |x| \le R$}\\
            0 & \text{ if $|x| \ge R$}.
        \end{cases*}
    \end{equation*}
    Hence, by \eqref{def: cap pot of truncated ball in Uezr} and \eqref{def: mueM}, we have
    \begin{equation}
        \label{eq: explicit formula for mueM}
        \begin{split}
            \mu^{\varepsilon}_{M} = \sum_{z \in I^{\varepsilon}_{g, M}}&\left( 1 - \left(\frac{\varepsilon^{\frac{d}{d - 2}}(\rho_{z} \wedge M)}{d^{\varepsilon}_{z}}\right)^{d - 2} \right)^{-1}
            \left( \frac{\varepsilon^{\frac{d}{d - 2}}(\rho_{z} \wedge M)}{d^{\varepsilon}_{z}} \right)^{d - 2}\frac{d - 2}{d^{\varepsilon}_{z}}\delta_{\partial U^{\varepsilon}_{z}}.
        \end{split}
    \end{equation}
    Therefore, 
    \begin{equation}
        0 \le \hat{\mu}^{\varepsilon}_{M} \le \mu^{\varepsilon}_{M} \le (d - 2)(1 - \varepsilon^{\frac{2}{d - 2}}M^{2})M^{d - 2}\varepsilon^{d}\sum_{z \in I^{\varepsilon}_{g, M}}\frac{\delta_{\partial U^{\varepsilon}_{z}}}{(d^{\varepsilon}_{z})^{d - 1}}
    \end{equation}
    whenever $\varepsilon^{\frac{2}{d - 2}}M^{2} \le 1$. Note that the right-hand side of the last display converges strongly in $H^{-1}(\R^{d})$ as $\varepsilon \searrow 0$ by Lemma \ref{thm: giunti 5_3}. Hence, by \cite[Lemma 2.8]{cioranescu1997strange} and the fact that $\mathrm{spt}\hat{\mu}^{\varepsilon}_{M} \sub \{d(\cdot, W) < 1\}$, it holds that with probability one, every subsequence of $(\hat{\mu}^{\varepsilon}_{M})_{\varepsilon \in ]0, 1]}$ is strongly relatively compact in $H^{-1}(\R^{d})$. From \eqref{eq: egM, eg to 0 in weak H1}, the compactness of $(\hat{\mu}^{\varepsilon}_{M})$, and Lemma \ref{thm: weak convergence of the modulus of grad egM} below, \eqref{eq: mueM converges} follows(cf. \cite[Proposition 1.1]{cioranescu1997strange}). Indeed, fix $M > 1$ and a realization $\omega$ such that $(\hat{\mu}^{\varepsilon}_{M}(\omega))_{\varepsilon \in ]0, 1]}$ is compact in $H^{-1}(\R^{d})$ and the convergences \eqref{eq: egM, eg to 0 in weak H1} and \eqref{eq: weak convergence of the modulus of grad egM} take place. Let a subsequence $(\varepsilon')$ and $\hat{\mu}_{M} \in H^{-1}(\R^{d})$ such that $\hat{\mu}^{\varepsilon'}_{M} \to \hat{\mu}_{M}$ in $H^{-1}(\R^{d})$. Then for every $\varphi \in C_{c}^{\infty}(\R^{d})$, 
    \begin{equation}
        \begin{split}
            &\int_{\R^{d}}|\nabla \hat{e}^{\varepsilon'}_{g, M}|^{2}\varphi\\
            &= \int_{\R^{d}}\nabla \hat{e}^{\varepsilon'}_{g, M}\cdot \nabla (\varphi (\hat{e}^{\varepsilon'}_{g, M} - 1)) + \int_{\R^{d}}\nabla \hat{e}^{\varepsilon'}_{g, M} \cdot \nabla \varphi - \int_{\R^{d}} \hat{e}^{\varepsilon'}_{g, M} \nabla \hat{e}^{\varepsilon'}_{g, M} \cdot \nabla \varphi\\
            &= \inn{-\hat{\mu}^{\varepsilon'}_{M}}{\varphi(\hat{e}^{\varepsilon'}_{g, M} - 1)}_{H^{-1}(\R^{d}), H^{1}(\R^{d})} + \int_{\R^{d}}\nabla \hat{e}^{\varepsilon'}_{g, M} \cdot \nabla \varphi - \int_{\R^{d}} \hat{e}^{\varepsilon'}_{g, M} \nabla \hat{e}^{\varepsilon'}_{g, M} \cdot \nabla \varphi\\
            &\to \inn{\hat{\mu}_{M}}{\varphi}_{H^{-1}(\R^{d}), H^{1}(\R^{d})}
        \end{split}
    \end{equation}
    by \eqref{eq: egM, eg to 0 in weak H1}. Hence, the convergence \eqref{eq: mueM converges} holds for this $\omega$.

    To show \eqref{eq: lim mueM to mu}, note first that
    \[\int_{\R_{\ge 0}^{2}} k \,\psi_{\frac{2}{M}}[N](dk, d\rho) \lesssim \int_{\R_{\ge 0}}\rho^{d - 2}\,\psi_{\frac{2}{M}}[\Xi](d\rho),\]
    since marks of $Z^{\varepsilon}$ belong to $\mathcal{Q}$. Thus, 
    \[\int_{\R_{\ge 0}^{2}}\left( \frac{M}{\rho \vee M} \right)^{d - 2}k \,\psi_{\frac{2}{M}}[N](dk, d\rho) \lesssim \int_{\R_{\ge 0}}\rho^{d - 2}\,\psi_{\frac{2}{M}}[\Xi](d\rho) \to 0\]
     a.s. as $M \to \infty$ by Lemma \ref{thm: continuity of thinning}. By the last estimate and the monotone convergence theorem, we have 
     \begin{equation*}
        \begin{split}
            &\int_{\R_{\ge 0}^{2}}\left( \frac{M}{\rho \vee M} \right)^{d - 2}k \,\psi^{\frac{2}{M}}[N](dk, d\rho)\\
            &= \int_{\R_{\ge 0}^{2}}\left( \frac{M}{\rho \vee M} \right)^{d - 2}k \,\psi[N](dk, d\rho) - \int_{\R_{\ge 0}^{2}}\left( \frac{M}{\rho \vee M} \right)^{d - 2}k \,\psi_{\frac{2}{M}}[N](dk, d\rho)\\
            &\to \int_{\R_{\ge 0}^{2}}k \,\psi[N](dk, d\rho) \text{ a.s.}
        \end{split}
     \end{equation*}
    The verification of \eqref{eq: nabla eg - egM} is an easy adaptation of the arguments in \cite{giunti2018homogenization} used to show an analogous estimate (the equation (4.70) in \cite{giunti2018homogenization}). For the sake of completeness, we give the proof here.
    
    Calculate 
    \begin{equation*}
        \begin{split}
            \n{\nabla \hat{e}^{\varepsilon}_{g} - \nabla \hat{e}^{\varepsilon}_{g, M}}_{L^{2}(W)}^{2} &\le \sum_{z \in I^{\varepsilon}_{g, M}}\int_{U^{\varepsilon}_{z}}|\nabla \hat{e}^{\varepsilon}_{z} - \hat{e}^{\varepsilon}_{z, M}|^{2} 
            + \sum_{z \in I^{\varepsilon}_{g} \setminus I^{\varepsilon}_{g, M}}\capa{\varepsilon z + \varepsilon^{\frac{d}{d - 2}}K^{\varepsilon}_{z}}{U^{\varepsilon}_{z}}\\
            &\lesssim \left(\sum_{\substack{z \in I^{\varepsilon}_{g, M},\\ \rho_{z} \ge M}} + \sum_{z \in I^{\varepsilon}_{g} \setminus I^{\varepsilon}_{g, M}}\right)\capa{\varepsilon z + \varepsilon^{\frac{d}{d - 2}}K^{\varepsilon}_{z}}{U^{\varepsilon}_{z}}\\
            &\le \left(\sum_{\substack{z \in I^{\varepsilon}_{g, M},\\ \rho_{z} \ge M}} + \sum_{z \in I^{\varepsilon}_{g} \setminus I^{\varepsilon}_{g, M}}\right)\capa{B_{\varepsilon^{\frac{d}{d - 2}}\rho_{z}}(\varepsilon z)}{U^{\varepsilon}_{z}}\\
            &\underset{\eqref{eq: evaluate dezr from below}}{\lesssim} \varepsilon^{d}\left(\sum_{\substack{z \in I^{\varepsilon}_{g, M},\\ \rho_{z} \ge M}} + \sum_{z \in I^{\varepsilon}_{g} \setminus I^{\varepsilon}_{g, M}}\right)\rho_{z}^{d - 2}\\
            &\lesssim \varepsilon^{d}\left( \sum_{\substack{ z \in \Phi \cap \varepsilon^{-1} W,\\ \rho_{z} \ge M}} + \sum_{z \in \mathcal{T}_{\frac{2}{M}}(\Phi) \cap \varepsilon^{-1} W} + \sum_{z \in \mathcal{T}^{\frac{2}{M}}(\Phi)\cap \varepsilon^{-1}W \setminus I^{\varepsilon}_{g, M}} \right)\rho_{z}^{d - 2}.
        \end{split}
    \end{equation*}
    Hence, 
    \begin{equation}
        \label{eq: limsup eg - egM to 0}
        \begin{split}
           &\limsup_{\varepsilon}\n{\nabla \hat{e}^{\varepsilon}_{g} - \nabla \hat{e}^{\varepsilon}_{g, M}}_{L^{2}(W)}^{2}\\
        &\lesssim |W|\left(\int_{\R_{\ge 0}^{2}}\rho^{d - 2}\chi_{\rho \ge M}\,\psi[N](dk, d\rho) + \int_{\R_{\ge 0}^{2}}\rho^{d - 2}\,\psi_{\frac{2}{M}}[N](dk, d\rho)\right)\\
        &\underset{M \to \infty}{\to} 0 \text{ a.s.},
        \end{split}
    \end{equation}
    where the first inequality follows from \eqref{eq: IegM to N2/M asymptotically}, \text{Lemma \ref{thm: averaging h over A}, and Lemma \ref{thm: giunti lem5_2}}, and the second from Lemma \ref{thm: continuity of thinning}.
\end{proof}

\begin{lemma}
    \label{thm: weak convergence of the modulus of grad egM}
    For every $M > 1$, a.s.,
    \begin{equation}
        \label{eq: weak convergence of the modulus of grad egM}
        |\nabla\hat{e}^{\varepsilon}_{g, M}|^{2}\,dx \to \chi_{W}\int_{\R_{\ge 0}^{2}}\left( \frac{M}{\rho \vee M} \right)^{d - 2}k \,\psi^{\frac{2}{M}}[N](dk, d\rho)
    \end{equation}
    weakly on $\R^{d}$.
    \end{lemma}

    \begin{proof}
        Let $A' \in \mathcal{B}(\R^{d})$ such that $|W \cap \partial A'| = 0$ and set $A := A' \cap W$. Then
        \begin{equation}
            \label{eq: evaluate the energy of hat egM on U from below}
            \begin{split}
                \int_{A'}|\nabla\hat{e}^{\varepsilon}_{g, M}|^{2} &\ge \int_{A}\sum_{\substack{z \in I^{\varepsilon}_{g, M},\\ \varepsilon z \in \{d(\cdot, \R^{d} \setminus A) > \varepsilon\}}}|\nabla \hat{e}^{\varepsilon}_{z, M}|^{2}\\
                &= \sum_{\substack{z \in I^{\varepsilon}_{g, M},\\ \varepsilon z \in \{d(\cdot, \R^{d} \setminus A) > \varepsilon\}}}\capa{\varepsilon z + \frac{M}{\rho_{z} \vee M}\varepsilon^{\frac{d}{d - 2}}K^{\varepsilon}_{z}}{U^{\varepsilon}_{z}}\\
                &\ge \varepsilon^{d}\sum_{\substack{z \in I^{\varepsilon}_{g, M},\\ \varepsilon z \in \{d(\cdot, \R^{d} \setminus A) > \varepsilon\}}}\left( \frac{M}{\rho_{z} \vee M} \right)^{d - 2}k_{z}.
            \end{split}
        \end{equation}
        Further, similarly to \eqref{eq: evaluate the difference between Ieg and N2/M}, we have
        \begin{equation}
            \begin{split}
                &\left\{ z \in \mathcal{T}^{\frac{2}{M}}(\Phi) \cap \varepsilon^{-1}A: z \notin I^{\varepsilon}_{g, M} \text{ or } d(\varepsilon z, \R^{d} \setminus A) \le \varepsilon \right\}\\
                &\sub \left\{ z \in \mathcal{T}^{\frac{2}{M}}(\Phi)\cap \varepsilon^{-1}W, d(\varepsilon z, D^{\varepsilon}_{b}) < \frac{\varepsilon}{M} \right\} \cup I^{\varepsilon}_{b}\\
                &\cup \left\{ z \in \Phi: \varepsilon z \in A \cap \{d(\cdot, \R^{d} \setminus A) \le \varepsilon\} \right\}.
            \end{split}
        \end{equation}
        By the last inclusion, \eqref{eq: points in thinned processs near Deb vanish asymptotically}, \eqref{eq: e^d Ieb to 0}, and the evaluation
        \begin{equation}
            \begin{split}
                &\varepsilon^{d}\#\left\{ z \in \Phi: \varepsilon z \in A \cap \{d(\cdot, \R^{d} \setminus A) \le \varepsilon\} \right\}\\
                &\le \liminf_{\delta \searrow 0}\limsup_{\varepsilon \searrow 0}\varepsilon^{d}\#\left\{ z \in \Phi: \varepsilon z \in \overline{A} \cap \{d(\cdot, \R^{d} \setminus A) \le \delta\} \right\}\\
                &\underset{\text{Lemma \ref{thm: averaging h over A}}}{\le} \liminf_{\delta \searrow 0}\left| \overline{A} \cap \{d(\cdot, \R^{d} \setminus A) \le \delta\} \right| \psi[N](\R_{\ge 0}^{2})\\
                &= |\partial A| \psi[N](\R_{\ge 0}^{2})\\
                &= 0 \text{ a.s.,}
            \end{split}
        \end{equation}
        we see that 
        \begin{equation}
            \varepsilon^{d}\#\left\{ z \in \mathcal{T}^{\frac{2}{M}}(\Phi) \cap \varepsilon^{-1}A: z \notin I^{\varepsilon}_{g, M} \text{ or } d(\varepsilon z, \R^{d} \setminus A) \le \varepsilon \right\} \to 0 \text{ a.s.}
        \end{equation}
        Therefore, by \eqref{eq: evaluate the energy of hat egM on U from below}, Lemma \ref{thm: giunti lem5_2}, and Lemma \ref{thm: averaging h over A},
        \begin{equation}
            \label{eq: liminf egM ge lim}
            \liminf_{\varepsilon}\int_{A'}|\nabla \hat{e}^{\varepsilon}_{g, M}|^{2} \ge |W \cap A'|\int_{\R_{\ge 0}^{2}}\left( \frac{M}{\rho \vee M} \right)^{d - 2}k\,\psi^{\frac{2}{M}}[N](dk, d\rho)
        \end{equation}
        a.s.

        Furthermore, 
        \begin{equation}
            \label{eq: evaluate the energy of egM on D from above}
            \begin{split}
                \int_{\R^{d}}|\nabla \hat{e}^{\varepsilon}_{g, M}|^{2} &\le \sum_{z \in I^{\varepsilon}_{g, M}}\capa{\varepsilon z + \frac{M}{\rho_{z} \vee M}\varepsilon^{\frac{d}{d - 2}}K^{\varepsilon}_{z}}{U^{\varepsilon}_{z}}\\
                &\le \sum_{z \in I^{\varepsilon}_{g, M}}\left( 1 + f_{d}\left( \frac{d^{\varepsilon}_{z}}{\varepsilon^{\frac{d}{d - 2}}M} \right) \right)\mathrm{Cap}\left( \frac{M}{\rho_{z} \vee M}\varepsilon^{\frac{d}{d - 2}}K^{\varepsilon}_{z} \right)\\
                &\le (1 + f_{d}(M^{-2}\varepsilon^{-\frac{2}{d - 2}}))\varepsilon^{d}\sum_{z \in \mathcal{T}^{\frac{2}{M}}(\Phi) \cap \varepsilon^{-1}W}\left( \frac{M}{\rho_{z} \vee M} \right)^{d - 2}k_{z},
            \end{split}
        \end{equation}
        where the second inequality follows from Lemma \ref{thm: Maz'ya lemma} and the fact that $\varepsilon z + \frac{M}{\rho_{z} \vee M}\varepsilon^{\frac{d}{d - 2}}K^{\varepsilon}_{z}$ is contained in a closed ball of radius $\varepsilon^{\frac{d}{d - 2}}M$, and the third inequality from the monotonicity of $f_{d}$ and the definition \eqref{def: IegM} of $I^{\varepsilon}_{g, M}$. Since $f_{d}(x) \to 0$ as $x \to +\infty$, by Lemma \ref{thm: averaging h over A} and \eqref{eq: evaluate the energy of egM on D from above}, we have, a.s., 
        \begin{equation}
            \label{eq: limsup egM le lim}
            \limsup_{\varepsilon \searrow 0}\int_{\R^{d}}|\nabla \hat{e}^{\varepsilon}_{g, M}|^{2} \le |W|\int_{\R_{\ge 0}^{2}}\left( \frac{M}{\rho \vee M} \right)^{d - 2}k \,\psi^{\frac{2}{M}}[N](dk, d\rho).
        \end{equation}
        From \eqref{eq: liminf egM ge lim} and \eqref{eq: limsup egM le lim}, \eqref{eq: weak convergence of the modulus of grad egM} follows.
    \end{proof}

\section{Corrector results}
\label{sect: corrector}
As Cioranescu and Murat have shown in \cite[Section 3]{cioranescu1997strange}, the functions $\hat{e}^{\varepsilon}$ constructed in Lemma \ref{thm: construction of osc fct} serve as correctors.

\begin{proposition}
    \label{thm: corrector}
    In the same setting as Lemma \ref{thm: construction of osc fct}, we further assume that there exists $p > d$ such that $u \in W^{1, p}(D)$ a.s. Then a.s.,
    \begin{equation}
        \label{eq: corrector}
        \n{u^{\varepsilon} - (u - \hat{e}^{\varepsilon}u)}_{H^{1}_{0}(D)} \to 0.
    \end{equation}
    If $\E[C_{0}] < \infty$ and $b := \n{\n{u}_{L^{\infty}(D)}}_{L^{\infty}(\pr)} < \infty$, then the convergence \eqref{eq: corrector} also takes place in $L^{2}(\pr)$.
\end{proposition}

\begin{proof}
Put $w^{\varepsilon} := 1 - \hat{e}^{\varepsilon}$. To prove the a.s. convergence, it suffices to employ the method developed in \cite[Section 3]{cioranescu1997strange} for $\pr$-a.e. realization $\omega \in \Omega$: calculate
\begin{equation*}
    \begin{split}
        \int_{D}|\nabla (u^{\varepsilon} - w^{\varepsilon}u)|^{2} &= \int_{D}|\nabla u^{\varepsilon}|^{2} + \int_{D}|w^{\varepsilon}|^{2}|\nabla u|^{2} + \int_{D}|u|^{2}|\nabla w^{\varepsilon}|^{2}\\
        &- 2\int_{D}w^{\varepsilon}\nabla u^{\varepsilon} \cdot \nabla u - 2\int_{D}u\nabla u^{\varepsilon} \cdot \nabla w^{\varepsilon} + 2\int_{D}w^{\varepsilon}u\nabla w^{\varepsilon} \cdot \nabla u\\
        &= \inn{f^{\varepsilon}}{u^{\varepsilon}}_{H^{-1}, H^{1}_{0}} + \int_{D}|w^{\varepsilon}|^{2}|\nabla u|^{2} - \int_{D}w^{\varepsilon}\nabla (|u|^{2}) \cdot \nabla w^{\varepsilon} - \inn{\Delta w^{\varepsilon}}{|u|^{2}w^{\varepsilon}}_{H^{-1}, H^{1}_{0}}\\
        &-2\int_{D}w^{\varepsilon}\nabla u^{\varepsilon} \cdot \nabla u + 2\int_{D}u^{\varepsilon}\nabla u \cdot \nabla w^{\varepsilon} + 2\inn{\Delta w^{\varepsilon}}{uu^{\varepsilon}}_{H^{-1}, H^{1}_{0}}\\
        &+ 2\int_{D}w^{\varepsilon}u\nabla w^{\varepsilon} \cdot \nabla u
    \end{split}
\end{equation*}
and note that a.s., we may calculate the limit as $\varepsilon \searrow 0$ of every term in the right-hand side of the last display, which allows us to conclude that 
\[\int_{D}|\nabla (u^{\varepsilon} - w^{\varepsilon}u)|^{2} \to \inn{f}{u}_{H^{-1}, H^{1}_{0}} + \int_{D}|\nabla u|^{2} + C_{0}\int_{D}|u|^{2} - 2\int_{D}|\nabla u|^{2} - 2C_{0}\int_{D}|u|^{2} = 0 \text{ a.s.}\]

Next, we turn to the proof of the convergence in $L^{2}(\pr)$. By construction, $0 \le \hat{e}^{\varepsilon} \le 1$. Thus, 
\begin{equation}
    \label{eq: bound for corrector}
    \begin{split}
        \n{\nabla(u^{\varepsilon} - (u - \hat{e}^{\varepsilon}u))}_{L^{2}(D)}^{2} &\le C(d, D)\left(\sup_{\varepsilon}\n{f^{\varepsilon}}_{H^{-1}(D)}^{2} + \n{u\nabla \hat{e}^{\varepsilon} + \hat{e}^{\varepsilon}\nabla u}_{L^{2}(D)}^{2}\right)\\
        &\le  C(d, D)\left(\sup_{\varepsilon}\n{f^{\varepsilon}}_{H^{-1}(D)}^{2} + b\n{\nabla \hat{e}^{\varepsilon}}_{L^{2}(D)}^{2}\right).
    \end{split}
\end{equation}
Moreover, we estimate 
\begin{equation}
    \label{eq: energy estimate of osc fct}
    \begin{split}
        \n{\nabla \hat{e}^{\varepsilon}}_{L^{2}(D)}^{2} &\le  \n{\nabla \hat{e}^{\varepsilon}_{g}}_{L^{2}(W)}^{2} + \n{\nabla \hat{e}^{\varepsilon}_{b}}_{L^{2}(W)}^{2}\\
        &\le \sum_{z \in I^{\varepsilon}_{g}}\mathrm{Cap}(\varepsilon z + \varepsilon^{\frac{d}{d - 2}}K^{\varepsilon}_{z}, U^{\varepsilon}_{z}) + \mathrm{Cap}\left( \bigcup_{z \in I^{\varepsilon}_{b}}(\varepsilon z + \varepsilon^{\frac{d}{d - 2}}K^{\varepsilon}_{z}), D^{\varepsilon}_{b} \right)\\
        &\le f_{d}(2)\varepsilon^{d}\sum_{z \in \Phi \cap \varepsilon^{-1}W}\mathrm{Cap}K^{\varepsilon}_{z}\\
        &\lesssim \varepsilon^{d}\sum_{z \in \Phi \cap \varepsilon^{-1}W}k_{z}.
    \end{split}
\end{equation}
Now let us assume that $b, \E[C_{0}] < \infty$. Then by Lemma \ref{thm: averaging h over A}, the right-hand side of \eqref{eq: energy estimate of osc fct} converges in $L^{1}(\pr)$ as $\varepsilon \searrow 0$. Hence, by \eqref{eq: bound for corrector}, \eqref{eq: energy estimate of osc fct}, and the a.s. convergence \eqref{eq: corrector}, we conclude that the convergence \eqref{eq: corrector} takes place in $L^{2}(\pr)$.
\end{proof}

\section{Appendix}
\label{sect: appendix}
This section contains several auxiliary results on MPPs. We first establish a strong law of large numbers (Lemma \ref{thm: averaging h over A}), which is a slight generalization of \cite[Theorem 2.10]{Bas25} in terms of the integrability of the limit. Next, using Lemma \ref{thm: averaging h over A}, we prove lemmas (Lemma \ref{thm: giunti lem5_2} and \ref{thm: giunti 5_3}) corresponding to \cite[Lemma 5.2]{giunti2018homogenization} and \cite[Lemma 5.3]{giunti2018homogenization}, the proofs of which in \cite{giunti2018homogenization} rely on a result that is not applicable in our setting (see \cite[Lemma 6.1]{giunti2018homogenization}). Lemma \ref{thm: continuity of thinning} is a variant of \cite[Proposition 7.4]{Bas25} adapted to our needs.

Fix a complete separable metric space $\mathcal{K}$ and a stationary MPP $X$ on $\R^{d}$ with marks in $\mathcal{K}$ such that $\E[\#(X \cap A)] < \infty$ for every bounded $A \in \mathcal{B}(\R^{d} \times \mathcal{K})$. Let $\Phi$ denote the ground process of $X$ and write $k = k_{z}$ if $(z, k) \in \Phi$. Note that we do not assume that $\E[\#(\Phi \cap I)] < \infty$ for bounded $I \in \mathcal{B}(\R^{d})$.
\begin{lemma}[cf. {\cite[Theorem 2.10]{Bas25}}]
    \label{thm: averaging h over A}
    Let $h: \mathcal{K} \to [0, \infty[$ be Borel measurable and satisfy
    \[\int_{\mathcal{K}}h(k) \,\psi[X](dk) < \infty \text{ a.s.}\]
    Then a.s., for every bounded $B \in \mathcal{B}(\R^{d})$ with $|\partial B| = 0$, 
    \begin{equation}
        \label{eq: avg h over A}
        \varepsilon^{d}\sum_{z \in \Phi \cap \varepsilon^{-1}B}h(k_{z}) \to |B|\int_{\mathcal{K}} h(k)\,\psi[X](dk)
    \end{equation}
    as $\varepsilon \searrow 0$. Consequently, for every bounded $B \in \mathcal{B}(\R^{d})$ with $|\partial B| = 0$ and bounded continuous $V: B \to \R$, 
    \begin{equation}
        \label{eq: avg Vh over A}
        \varepsilon^{d}\sum_{z \in \Phi \cap \varepsilon^{-1}B}V(\varepsilon z)h(k_{z}) \to \int_{B}V \,dx \int_{\mathcal{K}}h(k) \,\psi[X](dk)
    \end{equation}
    a.s. Further, if $\E \int_{\mathcal{K}}h(k) \,\psi[X](dk) < \infty$, then the two convergences above take place also in $L^{1}(\pr)$.
\end{lemma}

\begin{proof}
    Put
    \[\mathscr{I}^{d}_{b} := \{\Pi_{i = 1}^{d}I_{i}: \text{ each $I_{i}$ is of the form $[a_{i}, b_{i}[$ with $-\infty < a_{i} < b_{i} < \infty$}\}.\]
    By using \cite[Theorem 25.14]{kallenberg2021foundations} instead of \cite[Proposition 12.2.II]{daley2008introduction} in the proof of \cite[Theorem 12.2.IV]{daley2008introduction}, we see that for every $I \in \mathscr{I}^{d}_{b}$ containing the origin, 
    \[n^{-d}\sum_{z \in \Phi \cap nI}h(k_{z}) \to |I|\int_{\mathcal{K}}h(k)\,\psi[X](dk)\]
    a.s., and if $\E [\int_{\mathcal{K}}h\,d\psi[X]] < \infty$, in $L^{1}(\pr)$.
    From this, we deduce that 
    \[\varepsilon^{d}\sum_{z \in \Phi \cap \varepsilon^{-1}I}h(k_{z}) \to |I|\int_{\mathcal{K}}h(k)\,\psi[X](dk)\]
    a.s., and if $\E [\int_{\mathcal{K}}h\,d\psi[X]] < \infty$, in $L^{1}(\pr)$ for every $I \in \mathscr{I}^{d}_{b}$ as $\varepsilon \searrow 0$ and that with probability one, for every bounded $B \in \mathcal{B}(\R^{d})$ with $|\partial B| = 0$, 
    \[\varepsilon^{d}\sum_{z \in \Phi \cap \varepsilon^{-1}B}h(k_{z}) \to |B|\int_{\mathcal{K}}h(k) \,\psi[X](dk)\]
    (see \cite[Section 7.1]{Bas25}). Thus, by the portmanteau theorem, the a.s. convergence \eqref{eq: avg Vh over A} takes place.

    As for the $L^{1}(\pr)$ convergence, note that for every bounded $B \in \mathcal{B}(\R^{d})$ with $|\partial B| = 0$, bounded continuous $V: B \to \R$, and $I \in \mathscr{I}^{d}_{b}$ such that $I \supseteq B$, we have
    \[\left| \varepsilon^{d}\sum_{z \in \Phi \cap \varepsilon^{-1}B}V(\varepsilon z)h(k_{z}) \right| \le \sup_{B}|V|\varepsilon^{d}\sum_{z \in \Phi \cap \varepsilon^{-1}I}h(k_{z}),\]
    and the right-hand side of the last display converges in $L^{1}(\pr)$ if $\E \int_{\mathcal{K}}h\,d\psi[X] < \infty$.
\end{proof}

\begin{lemma}[cf.{\cite[Lemma 5.2]{giunti2018homogenization}}]
    \label{thm: giunti lem5_2}
    Let $A \in \mathcal{B}(\R^{d})$ be bounded and $h: \mathcal{K} \to [0, \infty[$ a Borel function such that $\int h \,d\psi[X] < \infty$ a.s. For $\varepsilon > 0$, let $I^{\varepsilon} \sub \Phi \cap \varepsilon^{-1}A$ such that $\varepsilon^{d}\#I^{\varepsilon} \to 0$ a.s. Then
    \[\varepsilon^{d}\sum_{z \in I^{\varepsilon}}h(k_{z}) \to 0 \text{ a.s.}\]
\end{lemma}

\begin{proof}
    We may assume that $|\partial A| = 0$. Then 
    \begin{equation*}
        \begin{split}
            &\limsup_{\varepsilon}\varepsilon^{d}\sum_{k \in I^{\varepsilon}}h(k_{z})\\
            &\le \liminf_{M \to \infty}\limsup_{\varepsilon}\left( M\varepsilon^{d}\#I^{\varepsilon} + \varepsilon^{d}\sum_{\substack{z \in \Phi \cap \varepsilon^{-1}A}} h(k_{z})\chi_{h \ge M}(k_{z})\right)\\
            &= \liminf_{M \to \infty}|A|\int h\chi_{h \ge M}\,d\psi[X]\\
            &= 0 \text{ a.s.}
        \end{split}
    \end{equation*}
\end{proof}

\begin{lemma}[cf.{\cite[Lemma 5.3]{giunti2018homogenization}}]
    \label{thm: giunti 5_3}
    Let $h: \mathcal{K} \to \R$ be bounded Borel and $\int_{\mathcal{K}}h  \,d\psi[X] < \infty$, $A \in \mathcal{B}(\R^{d})$ bounded with $|\partial A| = 0$, and $I^{\varepsilon}(\omega) \sub \Phi(\omega) \cap \varepsilon^{-1}A$ such that
    \begin{equation}
        \label{cond: Ie to NA asymptotically}
        \varepsilon^{d}\#(\Phi \cap \varepsilon^{-1}A \setminus I^{\varepsilon}) \to 0 \text{ a.s.}
    \end{equation}
    For $z \in I^{\varepsilon}(\omega)$, let $r^{\varepsilon}_{z}(\omega)$ such that 
    \begin{equation}
        \label{cond: rezk is of order e}
        r^{\varepsilon}_{z}(\omega) \asymp \varepsilon,
    \end{equation}
    and that the balls $(U^{\varepsilon}_{z})_{z \in I^{\varepsilon}} := (U_{r^{\varepsilon}_{z}}(\varepsilon z))_{z \in I^{\varepsilon}}$ are disjoint. Then a.s., as $\varepsilon \searrow 0$, 
    \begin{equation}
        \label{eq: weak-star convergence in L infty}
        \varepsilon^{d}\sum_{z \in I^{\varepsilon}}h(k_{z})\frac{\chi_{U^{\varepsilon}_{z}}}{|U^{\varepsilon}_{z}|} \wst \chi_{A}\int_{\mathcal{K}}h(k) \,\psi[X](dk) \text{ in $L^{\infty}(\R^{d})$}
    \end{equation}
    and
    \begin{equation}
        \label{eq: strong convergence in W- 1infty loc}
        \varepsilon^{d}\sum_{z \in I^{\varepsilon}} h(k_{z})\frac{\delta_{\partial U^{\varepsilon}_{z}}}{\mathcal{H}^{d - 1}(\partial U^{\varepsilon}_{z})} \to \chi_{A}\int_{\mathcal{K}}h(k) \,\psi[X](dk) \text{ strongly in $W^{-1, q}(\R^{d})$}
    \end{equation}
    for every $1 < q \le \infty$. Here, $\mathcal{H}^{d - 1}$ denotes the $(d - 1)$-dimensional Hausdorff measure.
\end{lemma}

\begin{proof}
    Note that 
    \[\n{\varepsilon^{d}\sum_{z \in I^{\varepsilon}}h(k_{z})\frac{\chi_{U^{\varepsilon}_{z}}}{|U^{\varepsilon}_{z}|}}_{L^{\infty}(\R^{d})} \lesssim \n{h}_{\mathrm{sup}}\]
    for every $\varepsilon > 0$ a.s., and that a countable subset of $C^{1}_{c}(\R^{d})$ is dense in $L^{1}(\R^{d})$. Hence, to show \eqref{eq: weak-star convergence in L infty}, it suffices to show that for every $\zeta \in C^{1}_{c}(\R^{d})$, 
    \begin{equation}
        \label{eq: convergence of Linfty-L1 pairing}
        \varepsilon^{d}\sum_{z \in I^{\varepsilon}}h(k_{z})\fint_{U^{\varepsilon}_{z}}\zeta \,dx \to \int_{A}\zeta(z) \,dz \int_{\mathcal{K}}h(k) \,\psi[X](dk) \text{ a.s.}
    \end{equation}
    Fix $\zeta \in C^{1}_{c}(\R^{d})$. Then 
    \begin{equation}
        \label{eq: approximate rezk by e}
        \begin{split}
            &\left| \varepsilon^{d}\sum_{z \in I^{\varepsilon}}h(k_{z})\fint_{U^{\varepsilon}_{z}}\zeta \,dx - \varepsilon^{d}\sum_{z \in I^{\varepsilon}}h(k_{z})\fint_{U_{\varepsilon}(\varepsilon z)}\zeta \,dx \right|\\
            &\lesssim \left| \sum_{z \in I^{\varepsilon}}h(k_{z}) \int_{U_{\varepsilon}(\varepsilon z)}\left( \zeta\left(\frac{r^{\varepsilon}_{z}}{\varepsilon}(y - \varepsilon z) + \varepsilon z\right) - \zeta(y) \right)dy\right|\\
            &\lesssim \varepsilon^{d}\sum_{z \in I^{\varepsilon}}|h(k_{z})| \varepsilon\n{\nabla \zeta}_{\mathrm{sup}}\\
            &\le \varepsilon \n{\nabla \zeta}_{\mathrm{sup}}\cdot \varepsilon^{d}\sum_{z \in \Phi \cap \varepsilon^{-1}A}|h(k_{z})|\\
            &\to 0 \text{ a.s.}
        \end{split}
    \end{equation}
    by Lemma \ref{thm: averaging h over A}. Further, 
    \begin{equation}
        \label{eq: approximate integral over a ball by the value at the center}
        \begin{split}
            &\left| \varepsilon^{d}\sum_{z \in I^{\varepsilon}}h(k_{z})\fint_{U_{\varepsilon}(\varepsilon z)}\zeta \,dx - \varepsilon^{d}\sum_{z \in I^{\varepsilon}}h(k_{z})\zeta(\varepsilon z) \right|\\
            &\le \varepsilon^{d}\sum_{z \in I^{\varepsilon}}|h(k_{z})|\fint_{U_{\varepsilon}(\varepsilon z)}|\zeta(x) - \zeta(\varepsilon z)|\,dx\\
            &\le \varepsilon \n{\nabla \zeta}_{\mathrm{sup}}\cdot \varepsilon^{d}\sum_{z \in \Phi \cap \varepsilon^{-1}A}|h(k_{z})|\\
            &\to 0 \text{ a.s.}
        \end{split}
    \end{equation}
    by Lemma \ref{thm: averaging h over A}. Also, by the assumption \eqref{cond: Ie to NA asymptotically},
    \begin{equation}
        \label{eq: compare sums over Ie and over N cap e-1A}
        \begin{split}
            \left|\varepsilon^{d}\left(\sum_{z \in I^{\varepsilon}} - \sum_{z \in \Phi \cap \varepsilon^{-1}A} \right)h(k_{z})\zeta(\varepsilon z) \right| 
            &\lesssim \n{h \zeta}_{\mathrm{sup}}\varepsilon^{d}\#\left( \Phi \cap \varepsilon^{-1}A \setminus I^{\varepsilon} \right)\\
            &\to 0 \text{ a.s.}
        \end{split}
    \end{equation}
    By \eqref{eq: approximate rezk by e} to \eqref{eq: compare sums over Ie and over N cap e-1A}, and by Lemma \ref{thm: averaging h over A}, \eqref{eq: convergence of Linfty-L1 pairing} holds true. Hence, the weak-star convergence \eqref{eq: weak-star convergence in L infty} is established.

    Next, we show \eqref{eq: strong convergence in W- 1infty loc}. By Lemma \ref{thm: difference between Dirac measures on sphere and on ball} below, 
    \begin{equation}
        \label{eq: difference between averaging on sphere and on ball}
        \begin{split}
            &\n{\varepsilon^{d}\sum_{z \in I^{\varepsilon}}h(k_{z})\frac{\delta_{\partial U^{\varepsilon}_{z}}}{\mathcal{H}^{d - 1}\left(\delta_{\partial U^{\varepsilon}_{z}}\right)} - \varepsilon^{d}\sum_{z \in I^{\varepsilon}}h(k_{z})\frac{\chi_{U^{\varepsilon}_{z}}}{\left| U^{\varepsilon}_{z} \right|}}_{W^{-1, \infty}(\R^{d})}\\
        &\lesssim \varepsilon \n{h}_{\mathrm{sup}}.
        \end{split}
    \end{equation}
    By \eqref{eq: difference between averaging on sphere and on ball} and \eqref{eq: weak-star convergence in L infty}, \eqref{eq: strong convergence in W- 1infty loc} holds since $\bigcup_{\varepsilon \in ]0, 1], z \in I^{\varepsilon}}U^{\varepsilon}_{z}$ is relatively compact.
\end{proof}

\begin{lemma}
    \label{thm: difference between Dirac measures on sphere and on ball}
    Let $(U_{i})_{i = 1}^{l}$ be disjoint open balls with each $U_{i}$ of radius $r_{i}$, and let $a_{1}, \cdots, a_{l}  \in \R$. Then 
    \begin{equation}
        \label{eq: difference between Dirac measures on sphere and on ball}
        \n{\sum_{i = 1}^{l}a_{i}\left( \delta_{\partial U_{i}} - \frac{d}{r_{i}}\chi_{U_{i}} \right)}_{W^{-1, \infty}(\R^{d})} \le \max_{1 \le i \le l}|a_{i}|.
    \end{equation}
\end{lemma}

\begin{proof}
    For $i = 1, \cdots, l$, let $q_{i}$ be the solution to
    \begin{equation}
        \label{def: qi in terms of bvp}
        \begin{cases}
            \Delta q_{i} = \frac{d}{r_{i}} & \text{ in $U_{i}$}\\
            \frac{\partial q_{i}}{\partial n} = 1 & \text{ on $\partial U_{i}$}\\
            q_{i} = 0 & \text{ in $\R^{d} \setminus U_{i}$},
        \end{cases}
    \end{equation}
    that is, 
    \begin{equation}
        \label{eq: explicit formula for qi}
        q_{i}(x) = \begin{cases}
            \frac{1}{2r_{i}}\left( |x - x_{i}|^{2} - r_{i}^{2} \right) & \text{ if $|x - x_{i}| \le r_{i}$}\\
            0 & \text{ else,}
        \end{cases}
    \end{equation}
    where $x_{i}$ is the center of $U_{i}$. Then for every $\varphi \in \mathcal{D}(\R^{d})$,
    \[\inn{-\Delta q_{i}}{\varphi} = \int_{\R^{d}}\nabla q_{i} \cdot \nabla \varphi = \int_{U_{i}}\nabla q_{i} \cdot \nabla \varphi = \int_{\partial U_{i}}\varphi \,d\mathcal{H}^{d - 1} - \frac{d}{r_{i}}\int_{U_{i}}\varphi.\]
    Hence, 
    \[\delta_{\partial U_{i}} - \frac{d}{r_{i}}\chi_{U_{i}} = -\Delta q_{i} \text{ in $\mathcal{D}'(\R^{d})$}.\]
    Therefore, 
    \begin{equation}
        \label{eq: evaluate the difference between Dirac measures on sphere and ball}
        \begin{split}
            \n{\sum_{i = 1}^{l}a_{i}\left( \delta_{\partial U_{i}} - \frac{d}{r_{i}}\chi_{U_{i}} \right)}_{W^{-1, \infty}(\R^{d})} &= \n{\Delta \sum_{i = 1}^{l}a_{i}q_{i}}_{W^{-1, \infty}(\R^{d})}\\
            &\le \n{\sum_{i = 1}^{l}a_{i}\nabla q_{i}}_{L^{\infty}(\R^{d})}\\
            &\le \max_{1 \le i \le l}|a_{i}|\n{\nabla q_{i}}_{L^{\infty}(\R^{d})}\\
            &\le \max_{1 \le i \le l}|a_{i}|.
        \end{split}
    \end{equation}
\end{proof}

\begin{lemma}[cf. {\cite[Proposition 7.4]{Bas25}}]
    \label{thm: continuity of thinning}
    \hfill
    \begin{enumerate}[(1)]
        \item\label{item: monotonicity of psi delta} For every fixed $0 < \delta_{1} < \delta_{2} < \infty$,
        \[\psi_{\delta_{1}}[X] \le \psi_{\delta_{2}}[X] \le \psi[X] \text{ a.s.}\]
        \item\label{item: continuity of thinning} Let $(\delta)$ be a sequence decreasing to $0$ and suppose that $h: \mathcal{K} \to [0, \infty[$ is a Borel measurable function such that $\{k: h(k) \le M\} \sub \mathcal{K}$ is bounded for every $M \in \R_{\ge 0}$, and  $\int h \,d\psi[X] < \infty \text{ a.s.}$ Then
        \[\psi_{\delta}[X](\mathcal{K}), \int h \,d\psi_
        \delta[X] \to 0 \text{ a.s.}\]
        as $\delta \searrow 0$.
    \end{enumerate}
\end{lemma}

\begin{proof}
    \ref{item: monotonicity of psi delta}. Put $Q := [0, 1]^{d}$.  Since $\mathcal{T}_{\delta_{1}}(X) \sub \mathcal{T}_{\delta_{2}}(X) \sub X$, a.s., it holds that for every $n \ge 1$ and bounded $K \in \mathcal{B}(\mathcal{K})$,
    \[\frac{1}{|nQ|}\sum_{\substack{(z, k) \in \mathcal{T}_{\delta_{1}}(X),\\ z \in nQ}}\chi_{K}(k) 
    \le \frac{1}{|nQ|}\sum_{\substack{(z, k) \in \mathcal{T}_{\delta_{2}}(X),\\ z \in nQ}}\chi_{K}(k) 
    \le \frac{1}{|nQ|}\sum_{\substack{(z, k) \in X,\\ z \in nQ}}\chi_{K}(k).\]
    Thus, by Lemma \ref{thm: averaging h over A},
    \[\psi_{\delta_{1}}[X](K) \le \psi_{\delta_{2}}[N](K) \le \psi[N](K) \text{ a.s.}\]
    for every bounded $K \in \mathcal{B}(\mathcal{K})$.

    \ref{item: continuity of thinning}. First, we show that $\psi_{\delta}[X](K) \to 0$ a.s. for every bounded $K \in \mathcal{B}(\mathcal{K})$. Indeed, 
    \[\E[\psi_{\delta}[X](K)] = \E[\#\mathcal{T}_{\delta}(X) \cap (Q \times K)] \to 0\]
    as $\delta \searrow 0$ by dominated convergence. Hence, by \ref{item: monotonicity of psi delta}, $\psi_{\delta}[X](K) \to 0$ a.s.
    
    Therefore,
    \begin{equation}
        \begin{split}
            \limsup_{\delta \searrow 0}\int h \,d\psi_{\delta}[X] &\le \liminf_{M \to \infty} \limsup_{\delta \searrow 0} \left( M\psi_{\delta}[X](\{h \le M\}) + \int h \chi_{h \ge M}\,d\psi[X] \right)\\
            &= \liminf_{M \to \infty}\int h\chi_{h \ge M}\,d\psi[X]\\
            &\to 0
        \end{split}
    \end{equation}
    and
\begin{equation}
    \psi_{\delta}[X](\mathcal{K}) \le \psi_{\delta}[X](\{h \le 1\}) +\int_{\mathcal{K}}h \,d\psi_{\delta}[X] \to 0
\end{equation}
a.s.
\end{proof}

\begin{lemma}
    \label{thm: correspondence between N hash g and admissible sets}
    Define
    \begin{equation}
        \begin{split}
            \mathcal{N}^{\#g}_{\R^{d} \times \mathcal{K}} := &\{\mu: \text{ $\mu$ is a $\mathbb{Z}_{\ge 0} \cup \{\infty\}$-valued Borel measure on $\R^{d} \times \mathcal{K}$,}\\
            &\text{$\mu(A \times \mathcal{K}) < \infty$ for every bounded $A \in \R^{d}$,}\\
            &\text{$\mu(\{x\} \times \mathcal{K}) \in \{0, 1\}$ for every $x \in \R^{d}$}\}.
        \end{split}
    \end{equation}
    For $\tau \in \R^{d}$ and $\mu \in \mathcal{N}^{\#g}_{\R^{d} \times \mathcal{K}}$, define $S_{\tau}'\mu \in \mathcal{N}^{\#}_{\R^{d} \times \mathcal{K}}$ by 
    \begin{equation}
        (S_{\tau}'\mu)(A) = \mu(\{(x + \tau, k): (x, k) \in A\})
    \end{equation}
    for $A \in \mathcal{B}(\R^{d} \times \mathcal{K})$.
    Then the maps 
    \[\mathcal{N}^{\#g}_{\R^{d} \times \mathcal{K}} \ni N \mapsto \mathrm{spt}N \in \M^{d}_{\mathcal{K}}\]
    and
    \[\M^{d}_{\mathcal{K}} \ni Y \mapsto \sum_{p \in Y}\delta_{p} \in \mathcal{N}^{\#g}_{\R^{d} \times \mathcal{K}}\]
    are reciprocal each other. Further, if $N \in \mathcal{N}^{\#g}_{E \times \mathcal{K}}$ and $Y \in \M$ correspond via the maps above, then
    \begin{enumerate}[(1)]
        \item $N(A) = \#(Y \cap A)$ for every $A \in \mathcal{B}(\R^{d} \times \mathcal{K})$.
        \item $\mathrm{spt}N(\cdot, \mathcal{K}) = \mathrm{dom}Y$.
        \item For every $\tau \in \R^{d}$, $S_{\tau}'N$ corresponds to $S_{\tau}Y$.
    \end{enumerate}
\end{lemma}

\begin{proof}
    First, let $N \in \mathcal{N}^{\#g}_{\R^{d} \times \mathcal{K}}$. Then
    \[N = \sum_{0 \le i < N(\R^{d} \times \mathcal{K})}\delta_{(x_{i}, k_{i})}\]
    for some $(x_{i}, k_{i})_{i < N(\R^{d} \times \mathcal{K})}$ such that $\{x_{i}\}$ is locally finite and $x_{i} \neq x_{j}$ whenever $i \neq j$. Hence, $\{(x_{i}, k_{k})\}$ is locally finite, and thus is closed in $\R^{d} \times \mathcal{K}$. Therefore, $\mathrm{spt}N = \{(x_{i}, k_{i})\}$, and it is admissible in the sense of Definition \ref{def: mpp}. Further, 
    \[N = \sum_{0 \le i < N(\R^{d} \times \mathcal{K})}\delta_{(x_{i}, k_{i})} = \sum_{p \in \mathrm{spt}N}\delta_{p}.\]
    From the last equality, $N(A) = \#(A \cap \mathrm{spt}N)$ for every $A \in \mathcal{B}(\R^{d} \times \mathcal{K})$, and 
    \[S_{\tau}'N = \sum_{0 \le i < N(\R^{d} \times \mathcal{K})}\delta_{(x_{i} - \tau, k_{i})} = \sum_{p \in \{(x - \tau, k): (x, k) \in \mathrm{spt}N\}}\delta_{p} = \sum_{p \in S_{\tau}(\mathrm{spt}N)}\delta_{p}.\]

    Next, let $Y \in \M^{d}_{\mathcal{K}}$ and put $N := \sum_{p \in Y}\delta_{p}$. Then for every bounded $A \in \mathcal{B}(\R^{d})$, 
    \begin{equation}
        \begin{split}
            N(A \times \mathcal{K}) &= \#\{(x, k) \in Y: x \in A\}\\
            &= \#(A \cap \mathrm{dom}Y) < \infty.
        \end{split}
    \end{equation}
    Further, for every $x \in \R^{d}$, 
    \begin{equation*}
        N(\{x\} \times \mathcal{K}) = 
        \begin{cases}
            1 & \text{if $x \in \mathrm{dom}Y$}\\
            0 & \text{else}.
        \end{cases}
    \end{equation*}
    Hence, $N \in \mathcal{N}^{\#g}_{\R^{d} \times \mathcal{K}}$ and $\mathrm{spt}N(\cdot, \mathcal{K}) = \mathrm{dom}Y$. Since $N(A) = \#(Y \cap A)$ for every $A \in \mathcal{B}(\R^{d} \times \mathcal{K})$, $\mathrm{spt}N = Y$.
\end{proof}

\bibliography{biblio}
\bibliographystyle{abbrv}
\end{document}